\documentclass{amsart}
\usepackage{amsmath,amsthm, amscd, amssymb, amsfonts} 
\usepackage{mathrsfs}
\usepackage[english]{babel}
\usepackage[latin1]{inputenc}
\usepackage{graphicx}
\usepackage{enumerate}
\usepackage{esint}
\usepackage{color}
\usepackage{hyperref}
\usepackage{cleveref}

\newtheorem{thm}{Theorem}

\newtheorem{lem}{Lemma}
\newtheorem{prop}{Proposition}

\theoremstyle{remark}
\newtheorem{rem}{Remark}

\newcommand{\eps}{\epsilon}

\newcommand{\al}{\alpha}
\newcommand{\sig}{\sigma}
\newcommand{\te}{\theta}

\newcommand{\de}{\delta}

\newcommand{\ga}{\gamma}
\newcommand{\Ga}{\Gamma}
\newcommand{\be}{\beta}

\newcommand{\RR}{\mathbb{R}}

\newcommand{\RS}{\mathcal{R}}

\newcommand{\KS}{\mathcal{K}}
\newcommand{\RC}{\mathbf{R}}
\newcommand{\KC}{\mathbf{K}}
\newcommand{\JS}{\mathcal{J}}
\newcommand{\JC}{\mathbf{J}}

\title[Behaviour of Schrödinger Riesz transforms over smoothness spaces ]{Behaviour of Schrödinger Riesz transforms over smoothness spaces}

\author{B. Bongioanni, E. Harboure and P. Quijano}



\email{bbongio@santafe-conicet.gov.ar}
\address{Instituto de Matem\'atica Aplicada del Litoral CONICET-UNL, and Facultad de Ingenier\'ia Qu\'imica, UNL. Colectora Ruta Nac. N 168, Paraje El Pozo.}
\email{harbour@santafe-conicet.gov.ar}
\address{Instituto de Matem\'atica Aplicada del Litoral, CONICET-UNL. Colectora Ruta Nac. N 168, Paraje El Pozo.
	3000 Santa Fe, Argentina.}
\email{pquijano@santafe-conicet.gov.ar}
\address{Instituto de Matem\'atica Aplicada del Litoral, CONICET-UNL, and Facultad de Ingenier\'ia Qu\'imica, UNL. Colectora Ruta Nac. N 168, Paraje El Pozo.
	3000 Santa Fe, Argentina.}
\date{}

\begin{document}
	
	\begin{abstract}
	As it was shown by Shen, the  Riesz transforms associated to the Schrödinger operator $L=-\Delta + V$ are not bounded on $L^p(\mathbb{R}^d)$-spaces for all $p, 1<p<\infty$, under the only assumption that the potential satisfies a reverse Hölder condition of order $d/2$, $d\geq3$. Furthermore, they are bounded only for $p$ in some finite interval of the type $(1,p_0)$, so it can not be expected to preserve regularity spaces. In this work we search for some kind of minimal additional conditions on the potential in order to obtain boundedness on appropriate weighted $BMO$ type regularity spaces for all first and second order Riesz transforms, namely for the operators $\nabla L^{-1/2}$, $V^{1/2}L^{-1/2}$, $\nabla^2 L^{-1}$, $VL^{-1}$ and $V^{1/2}\nabla L^{-1}$. We also explore to what extent such extra conditions are also necessary.
	\end{abstract}
	
	\maketitle

\section{Introduction}\label{sec-intro}

	Given a non-negative potential $V$ we consider the Schr\"odinger type operator in $\RR^d$
	\[
	L= -\Delta  + V.
	\]
	The behaviour on function spaces of the associated Riesz transforms plays an important role in studying regularity properties for the solutions of either $Lu=f$ or $Lu=\nabla \cdot f$ in terms of the data.
	
	In this work we shall be concerned with all first and second order Riesz transforms: those being singular as $\mathcal{R}_1= \nabla L^{-1/2}$ and $\mathcal{R}_2= \nabla^2 L^{-1}$ as well as those involving the pontential, namely,  $V^{1/2} L^{-1/2}$, $ V L^{-1}$ and $ V^{1/2} \nabla L^{-1}$.
	
	
	Under the assumption that the potential $V$ belongs to the  reverse-Hölder class $RH_q$ with $q>d/2$, $d\geq 3$, that is, there exists $C$ such that
	
	\begin{equation}\label{eq:RH}
	\left(\frac{1}{|B|}\int_B V^q\right)^{{1}/{q}} \leq
	C \frac{1}{|B|}\int_B V,
	\end{equation}
	holds for every ball $B$ in $\RR^d$, $L^p$-inequalities were derived for all them by Shen in  \cite{shen}. Let us remind that the set of exponents $q$ such that a given function belongs to $RH_q$ is an open interval, so our assumption can be written as $V\in RH_{d/2}$.
	
	Later on, many authors have been concerned with boundedness results for the Schr\"odinger Riesz transforms acting on different spaces. For weighted $L^p$ spaces see for example \cite{BHS-classesofweights}, \cite{BCH1_MR3042701}, \cite{tangMR3365805}, \cite{Ly_secondorder_MR3488134} and \cite{bongioanni2019weighted}; for the behaviour on suitable Hardy spaces we refer to \cite{DZ-HSH-99}, \cite{T1BuiLiLy}, \cite{MR3002608} and \cite{BCH3_MR3541813}; as for the continuity on appropriate $BMO$ and regularity spaces, see for instance
	\cite{BHS-RieszJMAA}, \cite{MR3180934_stinga}, \cite{T1BuiLiLy} and \cite{bongioanni2019weighted}.

	A distinctive feature of the Schr\"odinger Riesz transforms is that, due to the mild assumption on the potential $V$, they are not bounded in all $L^p$, $1<p<\infty$. Rather, they are bounded just in some initial segment of $p$, that is, an interval of the form $(1, p_0)$ for some $p_0 < \infty$ and this range is known to be optimal. The case of $\mathcal{R}_1$ is the best, in the sense that it turns to be bounded in all $L^p$, $1<p<\infty$, imposing the stronger condition $q>d$ on the potential. Moreover, as Shen shows (see~\cite{shen}), it is in fact a Calderón-Zygmund operator so, in particular, it is also of weak type when $p=1$. In the other cases, boundedness in all $L^p$ holds under an even stronger assumption, for instance, when the potential belongs to $RH_\infty$, which means that the above inequality holds with the $L^\infty$-norm of $\chi_BV$ on the left hand side.
	
	A fundamental tool in Shen's approach is the critical radius function associated to the potential $V$ defined by
	\begin{equation}\label{defrho}
	\rho(x)=\sup\left\{r>0: \frac{1}{r^{d-2}} \int_{B(x,r)} V\leq 1 \right\},\,\,
	x\in\RR^d.
	\end{equation}
	
	Whenever $V$ satisfies a reverse-Hölder condition of order $q>d/2$ we have $0<\rho(x)<\infty$ for all $x\in \mathbb{R}^d$. Moreover, the function $\rho$ satisfies the following two inequalities
	\begin{equation}\label{rox_vs_roy}
	c_\rho^{-1} \rho(x)\left(1+\frac{|x-y|}{\rho(x)}\right)^{-N_0} \ \leq \ \rho(y) \ \leq \ c_\rho\,\rho(x)\left(1+\frac{|x-y|}{\rho(x)}\right)^{\frac{N_0}{N_0+1}},
	\end{equation}
	for all $x, y \in \RR^d$, and some positive constants $c_\rho$ and $N_0$ (see \cite{shen} Corollary~1.5).
		
	Concerning the behaviour of these operators on regularity spaces, it is known that $\mathcal{R}_1$, under the extra assumption $q>d$, preserves suitable smoothness spaces. 
	
	More precisely, for $0\leq\beta <1$ and a weight $w$, let us introduce the spaces $BMO^\beta_\rho (w)$ as those locally integrable functions satisfying
	\begin{equation}\label{osc_acotada}
	\int_B |f-f_B| \leq C \, w(B)\,|B|^{\be/d}\text{, \ \ \ for any ball } B,
	\end{equation}
	and
	\begin{equation}\label{prom_acotado}
	\int_B |f| \leq C \, w(B)\,|B|^{\be/d}\text{, \ \ \ for all } B=B(x,R)
	\text{, \ with } R\geq \rho(x).
	\end{equation}
	Here, as usual, $f_B$ stands for the average of $f$ over the ball $B$, and the norm $\|f\|_{BMO_{\rho}^{\beta}(w)}$ is defined as the maximum of the infima of the constants in~\eqref{osc_acotada} and~\eqref{prom_acotado}. Also, when $\beta=0$ we denote the space as $BMO_\rho(w)$ and when $w=1$ we just write $BMO_\rho^\beta$.
	
	In~\cite{MR3180934_stinga} the authors show boundedness of $\mathcal{R}_1$ over such spaces for the case $w=1$, while in~\cite{bongioanni2019weighted} sufficient conditions on the weight are obtained to guarantee continuity over $BMO^\beta_\rho(w)$ assuming $V\in RH_q$ with $q>d$. Nevertheless, in both results, the index of regularity is restricted to $0\leq \beta< 1-d/q$. 
	
	For the remaining operators very little is known about their behaviour on regularity spaces. As we mentioned, we can not expect to get boundedness under the only condition $V\in RH_{d/2}$. In fact, if they were bounded on $BMO_\rho$, they would map $L^\infty$ into the classical $BMO$ and, by interpolation, they should be bounded on all $L^p$. Therefore, a stronger condition is required. Let us remark that, to our knowledge, the only result in this direction appeared in~\cite{T1BuiLiLy} where the authors obtained regularity results in the unweighted case for the operator $V^\gamma L^{-\gamma}$, $0<\gamma\leq 1$, assuming $V\in RH_\infty$ and $|\nabla V(x)| \leq C/\rho(x)$.
	
	Our purpose in this work is to give sufficient conditions on $V$ in order to get boundedness on $BMO^\beta_\rho(w)$ spaces for all the first and second  order Schrödinger-Riesz transforms. In all cases we start with $V\in RH_{d/2}$ and we add a kind of local smoothness condition relative to the function $\rho$. For instance, for $\mathcal{R}_2$ which is known to be bounded on $L^p$ for $1<p\leq q$, we further assume that $V$ satisfies for some $\alpha>0$,
	\begin{equation}\label{eq-suavV}
	|V(x)-V(y)|\leq C\frac{|x-y|^{\alpha}}{\rho^{2+\alpha}(x)}, \ \ \ \text{for $|x-y|<\rho(x)$.}
	\end{equation}
	Our results will show that, under these conditions,  $\mathcal{R}_2$ is bounded on $BMO^\beta_\rho$ for all $0\leq \beta < \alpha$ and moreover it remains true for a certain class of weights that will be introduced in the next section. In particular, if  \eqref{eq-suavVgamma} holds with $\alpha=1$ we obtain the whole range $0\leq \beta<1$.

	Let us remark that in \cite{shen} (see Corollary~2.8 and the proof of Theorem~0.3), it is showed that  $\mathcal{R}_2$ is bounded on $L^p$  for $1<p<\infty$ whenever $V\in RH_{d/2}$ and satisfies
	\begin{equation}\label{eq-tamV}
	V(x)\leq \frac{C}{\rho^2(x)}, \ \ \ \text{for all $x$.}
	\end{equation}
	As we shall see, that condition is weaker than $RH_\infty$ and it also looks like the limiting case of \eqref{eq-suavV} when $\alpha =0$. 
	
	Certainly, \eqref{eq-suavV} implies \eqref{eq-tamV}. In fact, if $V$ satisfies \eqref{eq-suavV}, given $x\in \RR^d$ and $y\in B=B(x,\rho(x))$ we have
	\begin{equation*}
	V(x) \leq  V(y) + \frac{C|x-y|^{\alpha}}{\rho^{2+\alpha}(x)} \leq  V(y) + \frac{C}{\rho^{2}(x)}.
	\end{equation*}
	Therefore, averaging over $B=B(x,\rho(x))$ in the $y$-variable
	\begin{equation*}
	V(x) = \frac{1}{|B|} \int_{B} V(x)dy  \leq \frac{1}{|B|} \int_{B} V + \frac{C}{\rho^{2}(x)} \leq \frac{C'}{\rho^{2}(x)},
	\end{equation*}
	where we have used the definition of $\rho(x)$.
	
	On the other hand, it is not hard to prove that $V\in RH_\infty$ implies $V\in RH_{d/2}$ and $V(x)\leq \frac{C}{\rho^2(x)}$. The first part is obvious and also, applying the $RH_\infty$ condition to balls of the type $B=B(x,\rho(x))$ we get
	\[
	V(x)\leq\sup_{B} V \leq C \frac{1}{|B|}\int_B V \leq \frac{C}{\rho^2(x)},
	\]
	by the definition of $\rho(x)$. However the reciprocal is not true since $V(x)=\frac{1}{1+|x|^{2-\varepsilon}}$ belongs to $RH_{d/2}$ but not to $RH_\infty$, otherwise, for $r>1$,
	\[
	1=V(0) \leq C\frac{1}{|B(0,r)|} \int_{B(0,r)} V(x)\,\, dx = c_{d,\varepsilon} r^{-2+\varepsilon},
	\] 
	which tends to zero when $r\rightarrow \infty$, for small $\varepsilon$. Moreover, elementary calculations show that \eqref{eq-tamV} holds for such potential $V$.

	As we said, for the operator $\mathcal{R}_1$, continuity on regularity spaces is known. Nevertheless, as we will show, such regularity results can be extended to all $0\leq \beta <1$, under the assumptions $V\in RH_{d/2}$ and $V$ satisfying \eqref{eq-tamV}. 
	
	We will also take care of those operators involving multiplication by a power of $V$. More generally, we will consider $T_\gamma= V^\gamma L^{-\gamma}$, $0<\gamma\leq 1$ and $S_\gamma= V^\gamma \nabla L^{-\frac{1}{2}-\gamma}$, for $0<\gamma\leq 1/2$. In such cases we shall prove boundedness results on weighted $BMO^\beta_\rho$, $0\leq \beta< \alpha$, as long as $V\in RH_{d/2}$ and
	\begin{equation}\label{eq-suavVgamma}
	|V^\gamma(x)-V^\gamma(y)|\leq C\frac{|x-y|^{\alpha}}{\rho^{2\gamma +\alpha}(x)}, \ \ \ \text{for $|x-y|<\rho(x)$.}
	\end{equation}
	Observe that if we take $\alpha=1$ we can conclude the same boundedness result as in~\cite{T1BuiLiLy} but with weaker hypotheses.
	
	Clearly, if \eqref{eq-suavV} is true, the last inequality also holds for any $0<\gamma\leq 1$ but the reciprocal is not necessarily true. However, as it is easy to check, inequality \eqref{eq-suavVgamma} implies also \eqref{eq-tamV}.
	
		The organization of the paper is as follows.

First, in Section~\ref{sec-pre} , we remind some properties of the function $\rho$, we introduce the needed classes of weights and state some known properties of the spaces $BMO^\beta_\rho(\omega)$. At this point we are ready to state our main results concerning the continuity on $BMO^\beta_\rho(\omega)$ of the operators $\mathcal{R}_1$, $\mathcal{R}_2$, $T_\gamma $ and $S_\gamma$ under appropriate additional conditions, which may change with the operator. Then we introduce what will be the essential tool in proving our main results, more precisely, a consequence of Theorem 2 in~\cite{bongioanni2019weighted}, that gives regularity results for a large class of  of operators that we call Schrödinger-Calderón-Zygmund operators of type $(\infty,\delta$).

To be able to apply such result we establish, in Section~\ref{sec-Gamma,} some better estimates of the fundamental solution of $L_{\tau}=-\Delta +V+i\tau$ under the stronger assumptions $V\in RH_{d/2}$ and \eqref{eq-tamV}. 

Next, the following two sections are devoted to prove the theorems. We show first, in Section~\ref{sec-reg-conV}, the results for the operators involving the potential, that is,  $T_\gamma= V^\gamma L^{-\gamma}$, $0<\gamma\leq1$ and $S_\gamma= V^\gamma \nabla L^{-\frac{1}{2}-\gamma}$, for $1/2<\gamma\leq 1$, while in Section~\ref{sec-reg-sing} continuity for the singular Riesz transforms, $\mathcal{R}_1$ and $\mathcal{R}_2$, is derived. In all cases we check that the requirements of the general theorem are satisfied, being the case of $\mathcal{R}_2$ the most delicate of all.

We finish our work with a digression in Section~\ref{sec-digression} about the necessity of the conditions imposed to the potential $V$. In fact, for small values of $\gamma$, we are able to prove that condition {\eqref{eq-suavVgamma} is necessary and sufficient for the $BMO^\alpha_\rho$-boundedness of $T_\gamma$.

\section{Preliminaries and statement of the main results}\label{sec-pre}

We consider, for a potential $V\in RH_{d/2}$, the associated critical radius function defined as in~\eqref{defrho}. As we mentioned (see~\cite{shen}), under this assumption we have $0<\rho(x)<\infty$ for all $x\in \mathbb{R}^d$ and \eqref{rox_vs_roy}.

Next, we state two easy consequences of property \eqref{rox_vs_roy} that will be used very often.

\begin{lem}\label{lem:mismacritica}
	Given $x$ and $y$ such that $|x-y|<k\rho(x)$, there exists a constant $C$ depending on $k$, $c_\rho$ and $N_0$, such that $\frac{\rho(x)}{C}\leq \rho(y) \leq C\rho(x)$.
\end{lem}
\begin{proof}
	The result is a direct consequence of \eqref{rox_vs_roy}.
\end{proof}

\begin{lem}\label{lem:sobrero}
	If $c_\rho$ and $N_0$ are the constants in \eqref{rox_vs_roy}, then
	\begin{equation}\label{roxaroy}
	1+\frac{|x-y|}{\rho(y)}\ \leq \ c_\rho \left(1+\frac{|x-y|}{\rho(x)}\right)^{N_0+1},
	\end{equation}
	for all $x, y \in \RR^d$.
\end{lem}
\begin{proof}
	If $x, y \in \RR^d$, from the left hand side of \eqref{rox_vs_roy} we have
	\begin{equation}\label{izqroxvsroy}
	\frac{1}{\rho(y)} \ \leq \ c_\rho \frac{1}{\rho(x)} \left(1+\frac{|x-y|}{\rho(x)}\right)^{N_0}.
	\end{equation}
	Then,
	\begin{equation*}
	\frac{|x-y|}{\rho(y)} \ \leq \ c_\rho \frac{|x-y|}{\rho(x)} \left(1+\frac{|x-y|}{\rho(x)}\right)^{N_0} \ \leq \ c_\rho \left(1+\frac{|x-y|}{\rho(x)}\right)^{N_0+1}.
	\end{equation*}
	Since,	$1 \leq \ \left(1+\frac{|x-y|}{\rho(x)}\right)^{N_0+1}$, inequality \eqref{roxaroy} follows.	
\end{proof}

Let us go back to the weighted regularity spaces presented in the introduction. We make two observations. First, since the condition on the averages \eqref{prom_acotado} is stronger than \eqref{osc_acotada} it is enough to require~\eqref{osc_acotada} for balls $B(x,r)$ with $r\leq \rho(x)$. We also remind that it is enough to ask inequality \eqref{prom_acotado} only for $R=\rho(x)$.

We also remind that for $0<\be\leq 1$, these spaces have a point-wise description that ensures some regularity condition at almost every point. In fact, if $w$ is doubling over balls $B(x,R)$ with $R\leq\rho(x)$, and defining for $x\in \RR^d$ and $r>0$,
\begin{equation*}
W_\be(x,r)= \int_{B(x,r)} \frac{w(z)}{|z-x|^{d-\be}}\, dz,
\end{equation*}
a function $f$ belongs to $BMO^\beta_\rho(w)$, if and only if 
\begin{equation}\label{cond_Lip_local}
\left|f(x)-f(y)\right|\ \leq\ C\, \left[ W_\be(x,|x-y|) +\ W_\be(y,|x-y|) \right]
\end{equation}
and
\begin{equation}\label{cond_Lip_global}
|f(x)|\ \leq\ C\, W_\be(x,\rho(x))
\end{equation}
for $x$ and $y$ in $\RR^d$ (see~\cite{BHSnegativepowers}). Notice that when $w\equiv1$, $W_\beta(x,r)=r^\beta$ and~\eqref{cond_Lip_local} is just the usual Lipschitz condition while~\eqref{cond_Lip_global} becomes $|f(x)|\leq C\rho(x)$.


To state our main results we need to introduce appropriate classes of weights associated to a critical radius function $\rho$. Following~\cite{BHS-classesofweights} we define, for a given $p> 1$,
$\displaystyle A_p^{\rho}=\bigcup_{\theta\geq0}A_p^{\rho,\te}$, where $A_p^{\rho,\te}$ is defined as those weights $w$ such that
\begin{equation}\label{eq:Ap}
\left(\int_B w\right)^{1/p}
\left(\int_B w^{-\tfrac{1}{p-1}}\right)^{1/p'}
\leq
C |B| \left(1+ \frac{r}{\rho(x)}\right)^{\te},
\end{equation} 
for every ball $B=B(x,r)$.

Similarly, when $p=1$, we denote $\displaystyle A_1^{\rho}=\bigcup_{\theta\geq0}A_1^{\rho,\te}$, where $A_1^{\rho,\te}$ is the class of weights $w$ such that
\begin{equation}\label{eq:A1}
\frac{1}{|B|}\int_B w \leq
C\left(1+\frac{r}{\rho(x)}\right)^{\theta}
\inf_{B} w,
\end{equation} 
for every ball $B=B(x,r)$. As in the classical Muckenhoupt theory, the $A_p^{\rho}$ classes are increasing with $p$, and their weights have a self-improvement property, that is, if $w\in A_p^{\rho}$ then $w\in A_{p-\eps}^{\rho}$ for some $\eps>0$ (see Proposition~5 in \cite{BHS-classesofweights}).
We will often use the notation $A_{\infty}^{\rho}=\bigcup_{p\geq1}A_{p}^{\rho}$.

Following the same lines, we define doubling classes of weights adapted to this context. For $\mu\geq 1$ let us denote $D_{\mu}^{\rho}=\bigcup_{\theta\geq0}D_\mu^{\rho,\te}$, where $D_\mu^{\rho,\te}$ is the class of weights $w$ such that
there exists a constant $C>0$ such that  
\begin{equation}\label{duplicacionMuRho}
w(B(x,R))\leq C w(B(x,r))\left(\frac{R}{r}\right)^{d\mu}\left(1+\frac{R}{\rho(x)}\right)^{\theta},
\end{equation}
for $x\in\RR^d$, $r\leq R$. For a weight $w\in A_p^{\rho}$ it easy to check that $w\in D_{p}^{\rho}$. 

Observe that for the Muckenhoupt classes we have, $A_p\subset A_p^\rho$ for $p\geq1$, and similarly $D_\mu\subset D_\mu^\rho$ for $\mu\geq 1$. However, the reciprocal is not true. For example, if $V\equiv 1$, the weight $w(x)=1+|x|^\sigma$ for $\sigma>d(p-1)$ belongs to $A_p^\rho$ but not to the classical $A_p$ space.

Now we are ready to state the main theorems we are going to prove. In the first two theorems we consider a generalization of the operators involving $V$ by studying the families $T_\ga=V^{\ga}L^{-\ga}$ with $0<\ga\leq 1$, and $S_\ga=V^\ga\nabla L^{-1/2-\ga}$ with $0<\ga\leq 1/2$. In fact, $T_{1/2}$, $T_1$ and $S_{1/2}$ give the operators named in the introduction.

\begin{thm}\label{thm:TGamma}
	Let $0<\gamma\leq 1$, $V\in RH_{d/2}$ and satisfying~\eqref{eq-suavVgamma} for some $0<\alpha\leq 1$. Then, the operator $T_\ga$ is bounded on $BMO_{\rho}^{\beta}(w)$, for $0\leq \beta<\alpha$, and any weight $w\in A_{\infty}^{\rho}\cap D_{\mu}^{\rho}$ and $1\leq \mu < 1+\frac{\alpha-\beta}{d}$.
\end{thm}

\begin{thm}\label{thm:SGamma}
	Let $0<\gamma \leq 1/2$, $V\in RH_{d/2}$ and satisfying~\eqref{eq-suavVgamma} for some $0< \alpha \leq 1$. Then, the operator $S_\gamma$ is bounded on $BMO^\beta_\rho (w)$ for $0\leq\beta < \alpha$, and any weight $w\in A^\rho_\infty \cap D_\mu^\rho$ such that $1\leq \mu <1+ \frac{\alpha-\beta}{d}$.
\end{thm}

For the singular Riesz transforms $\RS_1$ and $\RS_2$ we have the following results.

\begin{thm}\label{thm:R1nuevo}
	Let $V\in RH_{d/2}$  and such that  $V(x)\leq C/\rho^2(x)$. Then, for any 
	$0\leq\beta<1$, $\RS_1$ is bounded on $BMO_{\rho}^{\beta}(w)$ as long as $w\in A_{\infty}^{\rho}\cap D_{\mu}^{\rho}$, with $\displaystyle 1\leq \mu < 1+ \frac{1-\be}{d}$. 
\end{thm}

\begin{thm}\label{teo:R2BMO}
	Let $V\in RH_{d/2}$ and  satisfying \eqref{eq-suavV} for some $0<\alpha\leq 1$. If  $0\leq \beta < \alpha$, $\RS_2$ is bounded on $BMO^{\beta}_{\rho}(w)$ as long as $w\in A_\infty^{\rho}\cap D_{\mu}^{\rho}$, with $1\leq\mu<1+\frac{\alpha-\beta}{d}$.
\end{thm}

Let us notice that the condition imposed to $V$ in Theorem~\ref{thm:R1nuevo} is the weakest and moreover we obtain the whole range $0\leq\beta<1$. The strongest condition, as it may be expected, is the one for $\RS_2$, $T_1$ and $S_1$, that is, all the second order Schr\"odinger Riesz transforms. 

Condition~\eqref{eq-tamV} may be seen as the limiting case when $\gamma\rightarrow 0$ of~\eqref{eq-suavVgamma}, for a fixed $\alpha$. On the other hand, $\RS_1$ can be obtained taking $\gamma=0$ in the expression of $S_\gamma$. Therefore, taking $\alpha=1$,  Theorem~\ref{thm:R1nuevo} can be understood as the limiting case of Theorem~\ref{thm:SGamma}.

Our main tool in proving these results is based on a general $T1$ type theorem given in~\cite{bongioanni2019weighted} that provides sufficient conditions on an operator to be bounded on $BMO^{\beta}_{\rho}(w)$. Before  stating this theorem we need to define a class of operators. For $0<\delta\leq1$ we will say that a linear operator $T$ is a Schr\"odinger-Calder\'on-Zygmund operator of type~$(\infty,\delta)$ if:

\begin{enumerate}
	\item[($I$)] $T$ is bounded from $L^1$ into $L^{1,\infty}$.
	\item[($II$)] $T$ has an associated kernel $K:\RR^d\times\RR^d\rightarrow\RR$, in the sense that
	\begin{equation*}
	Tf(x)=\int_{\RR^d} K(x,y)f(y)dy,\,\,\,\,
	f\in L_c^{\infty} \,\,\text{and a.e.}\,\,x\notin \text{supp}f.
	\end{equation*}
	Further,  
	for each $N>0$ there exists a constant $C_N$ such that
	\begin{equation}\label{TamPuntual}
	|K(x,y)| \leq
	\frac{C_N}{|x-y|^{d}} \left(1+ \frac{|x-y|}{\rho(x)}\right)^{-N},\,\,\, x\neq y, 
	\end{equation}		
	and there exists $C$ such that
	\begin{equation}\label{suav-puntual}
	|K(x,y)-K(x_0,y)|
	\leq C \frac{|x-x_0|^{\de}}{|x-y|^{d+\de}},\,\,\,\text{when}\,\,
	|x-y|>2|x-x_0|.
	\end{equation}
\end{enumerate}

It is worth noting that this definition is slightly different than the one presented in~\cite{bongioanni2019weighted}. In fact, the conditions proposed here are stricter, so any result given in~\cite{bongioanni2019weighted} for SCZO of type~$(\infty, \delta)$ remains true with this definition. We choose to work with this definition since it is easier to handle and all the operators involved in this paper will satisfy it. 

The precise statement we need is the following. 

\begin{thm}\label{teo-T1}
	Let $T$ be a Schr\"odinger-Calder\'on-Zygmund operator of type~$(\infty,\delta)$ such that for some $\al>0$,
	\begin{equation}\label{condT1}
	|T1(x)-T1(z)| \leq C \, \left(\frac{r}{\rho(x)}\right)^\al,
	\end{equation}
	for all $x$, $z\in \RR^d$ such that $|x-z|\leq\rho(x)/2$.
	
	Then, if $0\leq\be<\sig = \min\{\de,\al\}$, the operator $T$ is bounded on $BMO^\be_\rho(w)$  for any $w\in A_\infty^\rho\cap D^\de_\mu$ with $1\leq\mu< 1 + \frac{\sig-\be}{d}$.
\end{thm}

This result is contained in Corollary~4 of~\cite{bongioanni2019weighted} taking $s=\infty$, $\varepsilon=\al$ and checking that~\eqref{condT1} implies the $T1$ condition stated there.

\section{Estimates of the fundamental solution}\label{sec-Gamma}

For $\tau\in\RR$, let $\Gamma(x,y,\tau)$ and $\Gamma_0(x,y,\tau)$ be the fundamental solution of $-\Delta+V+i\tau$ and $-\Delta+i\tau$ respectively. 

Since we use techniques based on the representation of an operator in terms of $\Gamma$, we will need some new estimates, assuming stronger conditions on the potential $V$, namely, $V\in RH_{d/2}$ and satisfying \eqref{eq-tamV}. Under the only assumption $V\in RH_{d/2}$ it was shown in \cite{shen} (see Theorem~2.7 there) that for all $N>0$, there exists $C_N$ such that
\begin{equation}\label{estga}
|\Gamma(x,y,\tau)|\leq  \frac{C_N\left(1+\frac{|x-y|}{\rho(x)}\right)^{-N}}{(1+|\tau|^{1/2}|x-y|)^N |x-y|^{d-2}},
\end{equation}
for all $x$ and $y$.

In the next lemma we prove with the additional condition \eqref{eq-tamV} some improved estimates for the smoothness of $\Ga$, $\nabla \Ga$ and its difference with $\Ga_0$, the fundamental solution of $-\Delta +i\tau$. We will follow closely the techniques developed in \cite{shen} and sometimes we just indicate how the stated estimates are obtained from the stronger hypotheses. In what follows we will denote $\nabla_1$ to the gradient operator with respect to the first $d$-dimensional variable, and analogously we denote $\nabla_2$ when derivatives are taken with respect to the second one.

\begin{lem}\label{lem:Gamma}
	Let $V\in RH_{d/2}$, and suppose \eqref{eq-tamV}. Then the following estimates hold.
	\begin{enumerate}[i)]
		\item\label{item:lemgamma1} \label{suavgammatau} If $N>0$ there exists $C_N$ such that
		\begin{equation}\label{estgahmenosga}
		|\Gamma(x+h,y,\tau)-\Gamma(x,y,\tau)|\leq \frac{C_N|h|\left(1+\frac{|x-y|}{\rho(x)}\right)^{-N}}{(1+|\tau|^{1/2}|x-y|)^N |x-y|^{d-1}},
		\end{equation}
		for all $h$, $x$ and $y$, whenever $|h|<|x-y|/4$.
		
		\item\label{item:suavgradgammatau} If $N>0$ and $0<\de<1$, there exists $C_{N,\de}$ such that
		\begin{equation}\label{estgradgahmenosga}
		|\nabla_1\Gamma(x+h,y,\tau)-\nabla_1\Gamma(x,y,\tau)|\leq \frac{C_{N,\de}|h|^\de\left(1+\frac{|x-y|}{\rho(x)}\right)^{-N}}{(1+|\tau|^{1/2}|x-y|)^N |x-y|^{d-1+\de}},
		\end{equation}
		for all $h$, $x$ and $y$, whenever $|h|<|x-y|/4$.
		
		\item\label{item:estgamenosga0} If $N>0$ there exists $C_N$ such that
		\begin{equation}\label{estgamenosga0}
		|(\Gamma-\Gamma_0)(x,y,\tau)|\leq \left(\frac{|x-y|}{\rho(x)}\right)\frac{C_N}{(1+|\tau|^{1/2}|x-y|)^N |x-y|^{d-2}},
		\end{equation}
		for all $x$ and $y$ with $|x-y|\leq \rho(x)$.
		
		\item \label{tamagradgammatau-gammatau0} If $N>0$, there exists $C_N$ such that
		\begin{equation}\label{estgradgamenosga0}
		|\nabla_1(\Gamma-\Gamma_0)(x,y,\tau)|\leq \left(\frac{|x-y|}{\rho(x)}\right)\frac{C_{N}}{(1+|\tau|^{1/2}|x-y|)^N |x-y|^{d-1}},
		\end{equation}
		for all $x$ and $y$ with $|x-y|\leq \rho(x)$.
		
		\item \label{item:suavgraddifdedif} If $N>0$ and $0<\de<1$, there exists $C_{N,\de}$ such that
		\begin{equation}\label{estgraddifdedif}
		\begin{split}
		&|\nabla_1(\Gamma-\Gamma_0)(x+h,y,\tau)-\nabla_1(\Gamma-\Gamma_0)(x,y,\tau)| \\ & \hspace{90pt} \leq \left(\frac{|x-y|}{\rho(x)}\right) \frac{C_{N,\de}|h|^\de}{(1+|\tau|^{1/2}|x-y|)^N |x-y|^{d-1+\de}},
		\end{split}
		\end{equation}
		for all $h$, $x$ and $y$, whenever $|h|<|x-y|/4$ and $|x-y|\leq \rho(x)$.
	\end{enumerate}
 	Moreover,  in each of the above estimates  we can replace $\rho(x)$ by $\rho(y)$ on the right hand side of the inequality.
\end{lem}
\begin{proof}
	For \ref{suavgammatau}) we use inequality~(4.8) in \cite{shen} to assert that if $u$ is a solution of
	\begin{equation}\label{usolLtau}
	(-\Delta+V+i\tau)u = 0,
	\end{equation}
	on the ball $B=B(x,R)$, then if $z\in \frac{1}{2}B$, we have for some constant $C$,
	\begin{equation*}
	|\nabla u(z)| \leq C \sup_{B}|u| \left(\int_{B} \frac{V(\xi)}{|z-\xi|^{d-1}} d\xi+\frac{1}{R}\right).
	\end{equation*}
	
	Now, if we use the additional condition on $V$, we have for $\xi\in B$,
	\begin{equation*}
	V(\xi)\leq \frac{C}{\rho^2(\xi)}\leq \frac{C}{\rho^2(x)} \left(1+\frac{|x-\xi|}{\rho(x)}\right)^{2N_0} \leq \frac{C}{\rho^2(x)} \left(1+\frac{R}{\rho(x)}\right)^{2N_0}.
	\end{equation*}
	
	Then, 
	\begin{equation}\label{est_sup_gradu}
	\begin{split}
	|\nabla u(z)| & \leq \frac{C}{R} \left(1+\frac{R}{\rho(x)}\right)^{2N_0}\frac{R^2}{\rho^2(x)} \sup_{B}|u|\\ & \leq \frac{C}{R} \left(1+\frac{R}{\rho(x)}\right)^{2N_0+2}\sup_{B}|u|.
	\end{split}
	\end{equation}
	
	Now we fix $x$ and $y$ in $\RR^d$, with $x\neq y$, and consider the function $u(z)=\Ga(z,y,\tau)$, a solution of $-\Delta+V+i\tau$ on the ball $B(x,\frac{|x-y|}{2})$. If $0<h<\frac{|x-y|}{4}$, we have
	\begin{equation}\label{gahmenosga}
	|\Ga(x+h,y,\tau)-\Ga(x,y,\tau)| = |u(x+h)-u(x)|\leq h |\nabla u(z)|,
	\end{equation}
	for some $z\in B(x,\frac{|x-y|}{4})$ (in the segment with endpoints on $x$ and $x+h$). Then, we use \eqref{est_sup_gradu} with $R=\frac{|x-y|}{2}$ to get
	\begin{equation*}
	|\nabla u(\xi)| \leq \frac{C}{|x-y|} \left(1+\frac{|x-y|}{\rho(x)}\right)^{2N_0+2}\sup_{\xi\in B(x,\frac{|x-y|}{4}) }|\Ga(\xi,y,\tau)|.
	\end{equation*}
	Therefore, we can obtain \eqref{estgahmenosga} as consequence of \eqref{estga} by choosing $N$ large enough and the fact that $|\xi-y|\simeq |x-y|$ for $\xi\in B(x,|x-y|/4)$.
	
	To prove \ref{item:suavgradgammatau}) we use equation (4.7) in \cite{shen}. Let $u$ be a solution of \eqref{usolLtau} on the ball $B=B(x,2R)$, for some fixed $R$ and $\eta$ be a smooth function such that $\eta=1$ in $B(x,3R/2)$, $\eta=0$ outside $B$, $|\nabla\eta|\leq C/R$ and $|\Delta\eta|\leq C/R^2$ for some constant $C$. Then, for $z\in B(x,3R/2)$, we can write
	\begin{equation*}
	\begin{split}
	\nabla u(z) = & \int_{\RR^d}\nabla_1\Gamma_0(z,\xi,\tau)[-V(\xi)u(\xi)\eta(\xi) + u(\xi)\Delta\eta(\xi)]d\xi \\
	& \hspace{1cm} + 2 \int_{\RR^d}\nabla_1\nabla_2\Gamma_0(z,\xi,\tau)u(\xi)\nabla\eta(\xi)d\xi
	\end{split}
	\end{equation*}
	Therefore, if $0<|h|<R$,
	\begin{equation}\label{gradu}
	|\nabla u(x+h)-\nabla u(x)| \leq C \sup_B |u| \left(I_1\left[\frac{1}{\rho^2(x)}\left(1+\frac{R}{\rho(x)}\right)^{2N_0} + \frac{1}{R^2}\right] + \frac{I_2}{R}\right),
	\end{equation}
	where
	\begin{equation*}
	I_1 = \int_{B}|\nabla_1\Gamma_0(x+h,\xi,\tau)-\nabla_1\Gamma_0(x,\xi,\tau)|d\xi,
	\end{equation*}
	and
	\begin{equation*}
	I_2 = \int_{3R/2<|\xi-x|<2R}|\nabla_1\nabla_2\Gamma_0(x+h,\xi,\tau)-\nabla_1\nabla_2\Gamma_0(x,\xi,\tau)|d\xi
	\end{equation*}
	
	We then split $I_1 = \int_{B(x,2|h|)} + \int_{B\setminus B(x,2|h|)} = I_{11} + I_{12}$. For $I_{11}$ we use $|\nabla \Gamma_0(z,\xi,\tau)|\leq C |x-\xi|^{1-d}$  to obtain
	\begin{equation}\label{elh}
	I_{11} \leq C \int_{B(x,2|h|)} \frac{d\xi}{|x+h-\xi|^{d-1}} + \int_{B(x,2|h|)} \frac{dy}{|x-\xi|^{d-1}} \le C |h|,
	\end{equation}
	
	For $I_{12}$ we use the Mean Value Theorem and since $|\nabla_1^2 \Gamma_0(z,\xi,\tau)|\leq C|z-\xi|^{-d}$, we obtain,
	\begin{equation*}
	I_{12} \le C \int_{2|h|<|x-\xi|<2R} \frac{|h|}{|\zeta-\xi|^{d}}d\xi,
	\end{equation*}
	for some $\zeta$ in the segment from $x$ to $x+h$. Then,
	\begin{equation*}
	I_{12} \le C \int_{2|h|<|x-\xi|<2R} \frac{|h|^\de}{|x-\xi|^{d-1+\de}}d\xi \le C|h|^\de R^{1-\de}.
	\end{equation*}
	From last estimate and \eqref{elh} we have
	\begin{equation}\label{I1}
	I_1\le C |h|^\de R^{1-\de}.
	\end{equation}
	
	To deal with $I_2$ we apply again the Mean Value Theorem and that  $|\nabla_1^2\nabla_2\Gamma_0(z,\xi,\tau)|\le C|z-\xi|^{-d-1}$. Thus, 
	\begin{equation}\label{I2}
	I_2\le C |h| R^{-1}\leq  C |h|^\de R^{-\de}.
	\end{equation}
	
	Now, from \eqref{gradu}, \eqref{I1} and \eqref{I2}, we have
	\begin{equation*}
	|\nabla u(x+h)-\nabla u(x)| \le C \sup_B |u| \frac{|h|^\de}{R^{1+\de}} \left(1+\frac{R}{\rho(x)}\right)^{2N_0+2},
	\end{equation*}
	
	Now, if we consider the function $u(z)=\Gamma(z,y,\tau)$ in the ball $B(x,\frac{|x-y|}{4})$, from last inequality and~\eqref{estga}, we get \eqref{estgradgahmenosga}.
	
	The proof of \ref{item:estgamenosga0}) follows the lines of the proof of Lemma~4.5 in \cite{shen}. There, starting from the formula	\begin{equation}\label{formulaGaGa0}
	\Ga-\Ga_0 = - \int_{\RR^d} \Ga_0 V \Ga,
	\end{equation}
	(see page 540 in \cite{shen}) the author obtains
	\begin{equation*}
	|\Gamma(x,y,\tau)-\Gamma_0(x,y,\tau)|\leq I_1 + I_2 + I_3 + I_4,
	\end{equation*}
	where
	\begin{equation*}
	I_1 \leq \frac{C}{{(1+|\tau|^{1/2}|x-y|)^N |x-y|^{d-2}}} \int_{B(x,|x-y|/2)} \frac{V(z)}{|z-x|^{d-2}}dz,
	\end{equation*}
	\begin{equation*}
	I_2 \leq \frac{C}{{(1+|\tau|^{1/2}|x-y|)^N|x-y|^{d-2}}} \int_{B(y,|x-y|/2)} \frac{V(z)}{|z-x|^{d-2}}dz,
	\end{equation*}
	\begin{equation*}
	I_3 \leq \frac{C}{{(1+|\tau|^{1/2}|x-y|)^N}}\int_{|x-y|/2<|z-y|<\rho(y)/2} \frac{V(z)}{|z-y|^{2d-4}}dz,
	\end{equation*}
	and
	\begin{equation*}
	I_4 \leq \frac{C\rho^N(y)}{{(1+|\tau|^{1/2}|x-y|)^N}}\int_{|z-y|>\rho(y)/2} \frac{V(z)}{|z-y|^{2d-4+N}}dz,
	\end{equation*}
	for some $N$ large enough.
	
	Using \eqref{eq-tamV} and the fact that $\rho(x)\simeq\rho(z)$ for all $z\in B(x,|x-y|/2) \subset B(x,\rho(x))$ (see Lemma~\ref{lem:mismacritica}), we have
	\begin{equation*}
	\int_{B(x,|x-y|/2)} \frac{V(z)}{|z-x|^{d-2}}dz \leq C \left(\frac{|x-y|}{\rho(x)}\right)^2,
	\end{equation*}
	and thus, $I_1$, $I_2$ are bounded by
	\begin{equation*}
	\frac{\left(\frac{|x-y|}{\rho(x)}\right)^2}{{(1+|\tau|^{1/2}|x-y|)^N  |x-y|^{d-2}}}.
	\end{equation*}
	
	To deal with $I_3$, we have
	\begin{equation*}
	\begin{split}
	\int_{|x-y|/2<|z-y|<\rho(y)/2} \frac{V(z)}{|z-y|^{2d-4}}dz & \leq \frac{C}{\rho^2(x)} \int_{|x-y|}^{\rho(y)}t^{-d+3}dt\\
	&  \leq \frac{C}{\rho^2(x)|x-y|^{d-3}} \int_{0}^{\rho(y)}dt	
	\\ &\leq C\frac{\left(\frac{|x-y|}{\rho(x)}\right)}{|x-y|^{d-2}},
	\end{split}
		\end{equation*}
	since $d\geq 3$ and $|x-y|\leq\rho(x)\simeq\rho(y)$.
	
	The same estimate can be obtained for $I_4$, since by \eqref{eq-tamV} and \eqref{izqroxvsroy}
	\begin{equation*}
	V(z) \leq \frac{C}{\rho^2(z)} \leq \frac{C}{\rho^2(y)} \left(1+\frac{|y-z|}{\rho(y)}\right)^{2N_0}\leq \frac{C}{\rho^2(y)} \left(\frac{|y-z|}{\rho(y)}\right)^{2N_0},
	\end{equation*}
	for all $z$ such that $|z-y|>\rho(y)/2$. Then,
	\begin{equation*}
	\begin{split}
	\rho^N(y) & \int_{|z-y|>\rho(y)/2} \frac{V(z)}{|z-y|^{2d-4+N}}dz\\
	& \leq C\rho^{N-2N_0-2}(y)\int_{|z-y|>\rho(y)/2} \frac{dz}{|z-y|^{2d-4+N-2N_0}}\\
	& \leq C \rho^{-d+2}(y).
	\end{split}
	\end{equation*}
	Therefore, since $|x-y|<\rho(x)$,  $d\geq 3$ and $\rho(y)\simeq\rho(x)$, we have
	\begin{equation*}
	\rho^{-d+2}(y) \leq C \left(\frac{|x-y|}{\rho(x)}\right) \frac{1}{|x-y|^{d-2}}.
	\end{equation*}
	
	In order to prove \ref{tamagradgammatau-gammatau0}) we follow the proof of Lemma~5.8 in \cite{shen}. From \eqref{formulaGaGa0} we may write
	\begin{equation}\label{eq-graddif}
	\nabla_1(\Ga-\Ga_0)(x,y,\tau) = -\int_{\RR^d} \nabla_1\Ga_0(x,z,\tau). V(z) \Ga(z,y,\tau) \,dz.
	\end{equation}
	
	Then if $|x-y|<\rho(x)$, we have
	\begin{equation*}
	|\nabla_1(\Ga-\Ga_0)(x,y,\tau)| \leq I_1 + I_2 + I_3,
	\end{equation*}
	where
	\begin{equation*}
	I_1 \leq  \frac{C}{{(1+|\tau|^{1/2}|x-y|)^N |x-y|^{d-2}}} \int_{B(x,|x-y|/4)} \frac{V(z)}{|z-x|^{d-1}}dz,
	\end{equation*}
	\begin{equation*}
	I_2 \leq \frac{C}{{(1+|\tau|^{1/2}|x-y|)^N|x-y|^{d-1}}} \int_{B(y,|x-y|/4)} \frac{V(z)}{|z-x|^{d-2}}dz,
	\end{equation*}
	and
	\begin{equation*}
	I_3 \leq \frac{C}{{(1+|\tau|^{1/2}|x-y|)^N}}\int_{|z-y|>|x-y|/4} \frac{V(z)}{|z-y|^{2d-3}}\left(1+\frac{|z-x|}{\rho(x)}\right)^{-N}dz,
	\end{equation*}
	for some $N$ large enough. Thus using \eqref{eq-tamV}, we can perform similar calculations as in the proof of \ref{item:estgamenosga0}) to  get \eqref{estgradgamenosga0}.
	
	Finally, to obtain \ref{item:suavgraddifdedif}) we may use again \eqref{eq-graddif}. Then, as in the proof Lemma~4 in \cite{BHS-RieszJMAA}, if $N>0$, $0<\de<1$, $x$, $y$, and $h$, with $|h|<|x-y|/4$ and $|x-y|<\rho(x)$, we have
	\begin{equation*}
	|\nabla_1(\Gamma-\Gamma_0)(x+h,y,\tau)-\nabla_1(\Gamma-\Gamma_0)(x,y,\tau)| \leq C_N(I_1 + I_2 + I_3 + I_4),
	\end{equation*}
	where
	\begin{equation*}
	I_1 \leq \frac{C}{{(1+|\tau|^{1/2}|x-y|)^N |x-y|^{d-2}}} \int_{B(x,2h)} \frac{V(z)}{|z-x|^{d-1}}dz,
	\end{equation*}
	\begin{equation*}
	I_2 \leq C \frac{|h|^\de}{{(1+|\tau|^{1/2}|x-y|)^N|x-y|^{d-2}}} \int_{B(x,|x-y|)} \frac{V(z)}{|z-x|^{d-1+\de}}dz,
	\end{equation*}
	\begin{equation*}
	I_3 \leq C \frac{|h|}{{(1+|\tau|^{1/2}|x-y|)^N|x-y|^d}} \int_{|x-y|/2<|z-x|<2|x-y|} \frac{V(z)}{|z-y|^{d-2}}dz,
	\end{equation*}
	and
	\begin{equation*}
	I_4 \leq C \frac{|h|}{{(1+|\tau|^{1/2}|x-y|)^N}|x-y|}\int_{|z-y|>2|x-y|} \frac{V(z)}{|z-x|^{2d-2}}dz.
	\end{equation*}
	 Notice that, this time, the last integral is convergent. Therefore, in a similar way as before, we use \eqref{eq-tamV} to derive \eqref{estgraddifdedif}.
	
	Finally, the last assertion is a consequence of~Lemma~\ref{lem:mismacritica} and  Lemma~\ref{lem:sobrero}.
	
\end{proof}

\begin{rem}
We notice that in~\ref{item:estgamenosga0}) we can obtain the factor $\left(\frac{|x-y|}{\rho(x)}\right)^2$ instead of $\frac{|x-y|}{\rho(x)}$ on the right hand side of \eqref{estgamenosga0} in the case of dimension $d>4$. In fact, the only problem comes from the estimates of $I_3$ and $I_4$ where the function $|z-y|^{-2d+4}$ is integrable at infinity only if $d>4$. The given estimate for~\ref{item:estgamenosga0}) is sharp when $d=3$. Also, as it is clear from the poof, we may obtain the factor $\left(\frac{|x-y|}{\rho(x)}\right)^2$ in \ref{item:suavgraddifdedif}) for any dimension, while in \ref{tamagradgammatau-gammatau0}) it would require $d>3$. Nevertheless, the stated estimates are sufficient to our purpose.
\end{rem}

\section{Scrh\"odinger Riesz transforms involving $V$}\label{sec-reg-conV}

In this section we are going to prove Theorem~\ref{thm:TGamma} and Theorem~\ref{thm:SGamma}, concerning the continuity of $T_\gamma=V^{\gamma}L^{-\gamma}$ for $0<\gamma\leq 1$, and $S_\gamma=V^{\gamma}\nabla L^{-\gamma-1/2} $ for $0<\gamma\leq 1/2$, acting on $BMO_{\rho}^{\beta}(w)$ spaces.

Since both families of operators consist of an integral operator point-wisely multiplied by $V^\gamma$, it is natural to assume certain kind of smoothness for $V^\gamma$.

The path we will follow in proving the announced results is the same in both cases. First, we obtain some estimates for the kernels of the integral operators, that is either $L^{-\gamma}$ or $\nabla L^{-\gamma-1/2}$, and then we show that $T_\gamma$ and $S_\gamma$ are Schr\"odinger-Calder\'on-Zygmund operators. The proof of the theorems are completed checking that the $T1$ condition is also satisfied, since an application of Theorem~\ref{teo-T1} gives the desired results. Let us point out that the estimates given in Lemma~\ref{lem:Gamma} will be essential to check all the needed requirements.

\subsection{Operators $T_\gamma=V^{\gamma}L^{-\gamma}$ for $0<\gamma\leq 1$.} First we are going to show that, under the additional hypothesis on the potential $V$, these operators are Schr\"odinger-Calder\'on-Zygmund of type~$(\infty,\alpha)$. In order to do this we observe that $T_\gamma$ can be written as

$$T_{\gamma} f (x) = \int_{\RR^d} V^{\gamma}(x) \JS_{\gamma}(x,y) f(y)dy,$$
where $\JS_\gamma$ is the kernel associated to the fractional integral operator $L^{-\gamma}$. Hence, it is natural to look for properties of this kernel. We summarize them in the following lemma. Here we denote $\JC_\gamma(x-y)$ the kernel associated to $(-\Delta)^{-\gamma}$.

\begin{lem}\label{lem-Jgamma}

  Let $ V\in RH_{d/2}$ and satisfying $V(x) \leq C/\rho^2(x)$. Then if $0< \gamma \leq 1$, we have the following estimates
  \begin{enumerate}[(a)]
  	\item\label{item-Jgam-a}
  	For any $N\geq 0$ there exists $C_N$ such that
  	\begin{equation*}
  	|\JS_\gamma(x,y)| \leq \frac{C_N}{|x-y|^{d-2\gamma}}  \left(1+\frac{|x-y|}{\rho(x)}\right)^{-N}.
  	\end{equation*}
  	\item\label{item-Jgam-b}
  	For any $N\geq 0$ there exists $C_N$ such that, if $|h|\leq  |x-y|/2$
  	\begin{equation*}
  	|\JS_\gamma(x,y)-\JS_\gamma(x+h,y)| \leq C \frac{|h|}{|x-y|^{d-2\gamma+1}}  \left(1+\frac{|x-y|}{\rho(x)}\right)^{-N}.
  	\end{equation*}
  	\item\label{item-Jgam-c}
  	 For $|x-y|\leq \rho(x)$ we have
  	\begin{equation*}
  	|\JS_\gamma(x,y)-\JC_\gamma(x-y)| \leq C \left(\frac{|x-y|}{\rho(x)}\right)\frac{ 1}{|x-y|^{d-2\gamma}}  .
  	\end{equation*}
  	\item\label{item-Jgam-d} For $|x-y|\leq \rho(x)$ we have
 
  \begin{equation}
  |\nabla_1(\JS_\gamma(x,y)-\JC_\gamma(x-y))| \leq C \left(\frac{|x-y|}{\rho(x)}\right)\frac{ 1}{|x-y|^{d-2\gamma+1}}.
  \end{equation}
  	
  \end{enumerate}
\end{lem}

\begin{proof}
	For~\eqref{item-Jgam-a} we refer to Section 4.7 in~\cite{MR3180934_stinga}, in fact, we only need  $V\in RH_{d/2}$. To prove~\eqref{item-Jgam-b} it is enough to consider $x$, $h$ and $y\in\RR^d$ such that $|h|\leq |x-y|/4$. Using the functional calculus formula
	\begin{equation}\label{eq:formulaespectral}
	L^{-\gamma}=-\frac{1}{2\pi} \int_{-\infty}^{\infty}(-i\tau)^{-\gamma}(L+i\tau)^{-1}d\tau,
	\end{equation}
	we have that
	\begin{equation*}
\JS_\gamma(x,y)=-\frac{1}{2\pi}\int_{-\infty}^{\infty}(-i\tau)^{-\gamma}	 \Gamma(x,y,\tau) d\tau.
	\end{equation*}
	Now, applying part~\ref{item:lemgamma1}) of Lemma~\ref{lem:Gamma}, 
	\begin{equation*}
	|\JS_\gamma(x,y)-\JS_\gamma(x+h,y)|\leq C_N \frac{|h|}{|x-y|^{d-1}} \left(1+\frac{|x-y|}{\rho(x)}\right)^{-N}\int_{0}^\infty \frac{\tau^{-\gamma}}{(1+\tau^{1/2}|x-y|)^N}  d\tau.
	\end{equation*}
	Setting $s=\tau|x-y|^2$ and computing the last integral we arrive to the desired estimate. Inequalities~\eqref{item-Jgam-c} and~\eqref{item-Jgam-d} follow applying again the spectral formula~\eqref{eq:formulaespectral} for $L$ and $-\Delta$, together with Lemma~\ref{lem:Gamma}, parts \ref{item:estgamenosga0}) and \ref{tamagradgammatau-gammatau0}) respectively.

\end{proof}

\begin{prop}\label{prop-VLesSCZ}
	Let $V\in RH_{d/2}$ and  $0<\gamma\leq 1$. If for some $0<\alpha\leq 1$, $V$ satisfies~\eqref{eq-suavVgamma}, then $T_\gamma$ is a  Schr\"odinger-Calder\'on-Zygmund operator of type~$(\infty,\alpha)$.
\end{prop}
\begin{proof}
	Let $0<\gamma\leq 1$ and $V\in RH_{d/2}$ satisfying~\eqref{eq-suavVgamma} for some $\alpha>0$. First, notice that $T_\gamma$ is of weak type $(1,1)$ due to Theorem~1 and Theorem~5 from~\cite{bongioanni2019weighted}.

	Now we are going to show that the kernel $V^{\gamma}(x)\JS_{\gamma}(x,y)$ satisfies the size condition~\eqref{TamPuntual}.  Using that~\eqref{eq-suavVgamma} implies~\eqref{eq-tamV}, and part~\eqref{item-Jgam-a} of Lemma~\ref{lem-Jgamma} we have that for each $N>0$ there exists $C_N$ such that, 
	\begin{equation*}
	\begin{split}
	|V^{\gamma}(x)\JS_{\gamma}(x,y)| & \leq 
	\frac{C_N}{\rho^{2\gamma}(x)|x-y|^{d-2\gamma}} \left(1+\frac{|x-y|}{\rho(x)}\right)^{-N}
	\\ & \leq 
	\frac{C_N}{|x-y|^{d}} \left(1+\frac{|x-y|}{\rho(x)}\right)^{-N+2\gamma}. 
	\end{split}
	\end{equation*}

	To check the point-wise smoothness~\eqref{suav-puntual} we consider $|x-y|>2|x-x'|$ and write
	\begin{equation*}
	\begin{split}
	|V^{\gamma}(x)\JS_{\gamma}(x,y)-&V^{\gamma}(x')\JS_{\gamma}(x',y)| 
	\\ &  \leq V^{\gamma}(x)|\JS_{\gamma}(x,y)-\JS_{\gamma}(x',y)| + |\JS_\gamma(x',y)| |V^\gamma(x)-V^\gamma(x')|.
	\end{split}
	\end{equation*}
	
	For the first term we can use the smoothness of $\JS_\gamma$ given in part~\eqref{item-Jgam-b} of Lemma~\ref{lem-Jgamma} together with~\eqref{eq-tamV} to obtain
	\begin{equation*}
	\begin{split}
	V^{\gamma}(x)|\JS_{\gamma}(x,y)-\JS_{\gamma}(x',y)| & \leq \frac{C_N|x-x'|}{\rho^{2\gamma}(x)|x-y|^{d-2\gamma+1}} 
	\left(1+\frac{|x-y|}{\rho(x)}\right)^{-N}
	\\& \leq C_N \frac{|x-x'|}{|x-y|^{d+1}}
	\left(1+\frac{|x-y|}{\rho(x)}\right)^{-N+2\gamma}. 
	\end{split}
	\end{equation*}

	As for the second term we have two cases. If $|x-x'|\leq\rho(x)$ we can apply~\eqref{eq-suavVgamma}, Lemma~\ref{lem-Jgamma} to obtain
	\begin{equation*}
	\begin{split}
	|\JS_\gamma(x',y)| |V^\gamma(x)-V^\gamma(x')| & \leq C_N\frac{|x-x'|^{\alpha}}{|x'-y|^{d-2\gamma}\rho^{2\gamma+\alpha}(x)}\left(1+\frac{|x'-y|}{\rho(x')}\right)^{-N}
	\\ & \leq 
	C_N \frac{ |x-x'|^{\alpha} }{ |x-y|^{d-2\gamma} \rho^{2\gamma+\alpha}(x) }  \left(1+\frac{|x-y|}{\rho(x)}\right)^{-N}
	\\ & \leq 
	C_N \frac{|x-x'|^\alpha}{|x-y|^{d+\alpha}} 
	\left(1+\frac{|x-y|}{\rho(x)}\right)^{-N+2\gamma+\alpha},
	\end{split}
	\end{equation*}
	since $|x-y|\simeq|x'-y|$ and $\rho(x)\simeq\rho(x')$.
	Otherwise, if $|x-x'|> \rho(x)$ we use~\eqref{eq-tamV} and inequality~\eqref{rox_vs_roy} to get the bound
	\begin{equation}\label{eq-difVgamma}
	\begin{split}
	 |V^\gamma(x)-V^\gamma(x')| & \leq  C\left(\frac{1}{\rho^{2\gamma}(x)}+\frac{1}{\rho^{2\gamma}(x')}\right)
	 \\ & \leq \frac{C}{\rho^{2\gamma}(x)}\left(1+\frac{|x-x'|}{\rho(x)}\right)^{2\gamma N_0}
	  \\ & \leq \frac{C |x-x'|^\alpha}{\rho^{2\gamma+\alpha}(x)} \left(1+\frac{|x-x'|}{\rho(x)}\right)^{2\gamma N_0}.
	\end{split}
	\end{equation} 
	Then, we can proceed as for the other cases to obtain the desired estimate.
\end{proof}

With the aid of last proposition and Theorem~\ref{teo-T1} we are ready  to prove Theorem~\ref{thm:TGamma}.

\begin{proof}[Proof of Theorem~\ref{thm:TGamma}]
	By Proposition~\ref{prop-VLesSCZ} we know that $T_\gamma$ is a  Schr\"odinger-Calder\'on-Zygmund operator of type~$(\infty,\alpha)$. So, it remains to prove that $T_\gamma$ satisfies~\eqref{condT1} in order to apply Theorem~\ref{teo-T1}, obtaining, in this way, the desired result. To check that, observe first that $T_\gamma (1)$ is finite everywhere. Now let $x$, $z\in \RR^d$ such that $|x-z|\leq\rho(x)/2$, and write
	\begin{equation*}
	\begin{split}
	|V^{\gamma}L^{-\gamma}1 (x) - V^{\gamma}L^{-\gamma}1(z)| 
	& =\left| V^{\gamma}(x)\int_{\RR^d}\JS_\gamma (x,y)dy - V^\gamma(z)\int_{\RR^d}\JS_{\gamma}(z,y)dy\right|
	\\ & \leq 
	V^{\gamma}(x)\left| \int_{\RR^d} \left[\JS_{\gamma}(x,y)-\JS_{\gamma}(z,y)\right] dy\right|
	\\ &  \hspace{1cm}+
	|V^\gamma(x)-V^\gamma(z)|\int_{\RR^d} |\JS_\gamma(z,y)|dy  = A + B.
	\end{split}
	\end{equation*}
	
	To estimate $A$ we are going to introduce the kernel $\JC_\gamma(x-y)$ associated to the classical fractional integral $(-\Delta)^{-\gamma}$. Denoting $D_\gamma(x,y)=\JS_{\gamma}(x,y)-\JC_\gamma(x-y)$ we may write
	\begin{equation*}
	\begin{split}
	\int_{\RR^d} \left[\JS_{\gamma}(x,y)-\JS_{\gamma}(z,y)\right]dy & = \int_{B(x,\rho(x))}\left[ D_\gamma(x,y)-D_\gamma(z,y)\right]dy
	\\ & \;\;\;\;+
	\int_{B(x,\rho(x))} \left[\JC_\gamma(x-y)-\JC_\gamma(z-y)\right] dy
	\\ & \;\;\;\;\;\;\;\;
	+ \int_{B(x,\rho(x))^c } \left[\JS_\gamma(x,y)-\JS_\gamma(z,y)\right] dy 
	\\ &  = I + II + III.
	\end{split}
	\end{equation*}
	For $I$ we  decompose the integral in the following way
	\begin{equation*}
	\begin{split}
	|I| & \leq \int_{2|x-z|\leq|x-y|<\rho(x)}|D_\gamma(x,y)-D_\gamma(z,y)|dy
	\\& \;\;\;\;+ \int_{B(x,2|x-z|)}|D_\gamma(x,y)|dy + \int_{B(z,3|x-z|)}|D_\gamma(z,y)|dy = I_1 + I_2 + I_3.
	\end{split}
	\end{equation*}
	Then, for  $I_1$ we can apply part~\eqref{item-Jgam-d} of Lemma~\ref{lem-Jgamma} to obtain
	\begin{equation*}
	\begin{split}
	I_1  & \leq |x-z| \int_{|x-y|<\rho(x)}\left(\frac{|x-y|}{\rho(x)}\right) \frac{dy}{|x-y|^{d-2\gamma+1}}
	\\ & \leq C\frac{|x-z|}{\rho^{1-2\gamma}(x)}.
	\end{split}
	\end{equation*}
	As for $I_2$ we make use of part~\eqref{item-Jgam-c} of Lemma~\ref{lem-Jgamma} obtaining
	\begin{equation*}
	\begin{split}
	I_2  & \leq  \frac{C}{\rho(x)}\int_{|x-y|<2|x-z|} \frac{dy}{|x-y|^{d-2\gamma-1}}
	\\ & \leq C\frac{|x-z|}{\rho^{1-2\gamma}(x)},
	\end{split}
	\end{equation*}
	since $|x-z|\leq\rho(x)$. To estimate $I_3$ we proceed in a similar way.

	Now we are going to take care of $II$. We may write
	\begin{equation*}
	\begin{split}
	|II| & \leq \left|\int_{B(x,\rho(x))}\JC_\gamma(x-y)dy - \int_{B(z,\rho(x))}\JC_\gamma(z-y)dy\right|
	\\ & \;\;\;\;+\int_{B(x,\rho(x))\triangle B(z,\rho(x)) }|\JC_\gamma(z-y)|.
	\end{split}
	\end{equation*}
	Notice that the first term above equals zero. As for the second term, we have $|z-y|\simeq \rho(x)\simeq \rho(z)$ and 
	$$|B(x,\rho(x))\triangle B(z,\rho(x))|\simeq|x-z|\rho^{d-1}(x).$$ So,
	\begin{equation*}
	\begin{split}
	|II| \leq C \rho^{2\gamma-d}(x)|B(x,\rho(x))\triangle B(z,\rho(z))|
	\leq C \frac{|x-z|}{\rho^{1-2\gamma}(x)}.
	\end{split}
	\end{equation*}
	
	Now we turn our attention to $III$. Since $|x-z|<|x-y|/2$ and in view of Lemma~\ref{lem-Jgamma}, part~\eqref{item-Jgam-b}, we may write
	\begin{equation*}
	\begin{split}
	|III|&\leq C_N|x-z|\int_{B(x,\rho(x))^c }\frac{1}{|x-y|^{d-2\gamma+1}} \left(1+\frac{|x-y|}{\rho(x)}\right)^{-N}dy	  
	\leq C \frac{|x-z|}{\rho^{1-2\gamma}(x)},
	\end{split}
	\end{equation*}
	choosing $N>2\gamma-1$.
	
	Applying~\eqref{eq-tamV} and the above estimates we arrive to
	\begin{equation*}
	A\leq C V^\gamma(x) \frac{|x-z|}{\rho^{1-2\gamma}(x)} \leq C \frac{|x-z|}{\rho(x)}.
	\end{equation*}
	
	Finally, to bound $B$ we may use the size estimate for $\JS_\gamma$ given in Lemma~\ref{lem-Jgamma} and~\eqref{eq-suavVgamma} since $|x-z|\leq \rho(x)$,  obtaining
	\begin{equation*}
	\begin{split}
	B & \leq \frac{|x-z|^\alpha}{\rho^{2\gamma+\alpha}(z)} \int_{\RR^d} \frac{C_N}{|z-y|^{d-2\gamma}} \left(1+\frac{|z-y|}{\rho(z)}\right)^{-N}dy
	\\ & \leq C_N \frac{|x-z|^\alpha}{\rho^{2\gamma+\alpha}(z)}
	\left(
	\int_{|z-y|<\rho(z)}  \frac{dy}{|z-y|^{d-2\gamma}} +
	\rho^N(z) \int_{|z-y|\geq \rho(z)} \frac{dy}{|z-y|^{d-2\gamma + N}}
	\right)
	\\ & \leq 
	C \left(\frac{|x-z|}{\rho(z)}\right)^{\alpha},
	\end{split}
	\end{equation*}
	if we choose $N>2\gamma$. Since $\rho(z)\simeq\rho(x)$, inequality~\eqref{condT1} holds. A direct application of Theorem~\ref{teo-T1} completes the proof.

\end{proof}

\begin{rem}\label{rem-T1L}
	Notice that in the above proof  we have obtained for $0< 2\gamma <1$,
	\[
	|L^{-\gamma}1(x)-L^{-\gamma}1(z)|\leq C\left( \frac{|x-z|}{\rho(x)}\right)^{1-2\gamma} |x-z|^{2\gamma}
	\]
	as long as $|x-z|< \rho(x)/2$.
\end{rem}

\subsection{Operators $S_\gamma=V^{\gamma}\nabla L^{-\gamma-1/2}$ for $0<\gamma \leq 1/2$}

Now we analyse this family of operators, under the same conditions on $V$, to obtain similar results for continuity on $BMO^\beta_\rho (w)$.

As before we state a lemma with some technical estimates in order to check that $S_\gamma$ satisfies all the required conditions to apply Theorem~\ref{teo-T1}. Again our operator is the point-wise product of $V^\gamma$ times the integral operator $\nabla L^{-1/2-\gamma}$. So, we give first some estimates for its kernel. Notice that in the next lemma we include the case $\gamma=0$. Therefore, we are also  obtaining estimates for the kernel of the first order Riesz transform $\RS_1=\nabla L^{-1/2}$, that will be useful in the next section.

\begin{lem}~\label{lem-Hnu}
	Let $ V\in RH_{d/2}$ and satisfying $V(x) \leq C/\rho^2(x)$. Then, if $\mathcal{H}_\nu$ denotes the kernel of $\nabla L^{-\nu}$, $1/2\leq \nu \leq 1$, we have the following estimates
	\begin{enumerate}[(a)]
		\item\label{item-Hnu-a}
		For any $N\geq 0$ there exists $C_N$ such that
		\[
		|\mathcal{H}_\nu(x, y)| \leq\frac{C_N}{|x-y|^{d-2\nu +1}}\left( 1+ \frac{|x-y|}{\rho(x)}\right) ^{-N}.
		\]
		\item\label{item-Hnu-b}
		For any $0<\delta<1$  and $N\geq 0$ there exists $C_N$ such that if $|h|\leq  |x-y|/2$.
		\[
		|\mathcal{H}_\nu(x+h,y)- \mathcal{H}_\nu (x,y)|\leq\frac{C_N |h|^\delta}{|x-y|^{d-2\nu +1+\delta}}\left( 1+ \frac{|x-y|}{\rho(x)}\right) ^{-N}.
		\]
		\item\label{item-Hnu-c}
		
		If $\textbf{H}_\nu$ stands for the kernel of $\nabla (-\Delta)^{-\nu}$, for $|x-y|\leq \rho(x)$ we have
		\[
		|\mathcal{H}_\nu (x,y)- \textbf{H}_\nu (x,y)|\leq C \left( \frac{|x-y|}{\rho(x)}\right)\frac{1}{|x-y|^{d-2\nu +1}}.
		\]
		\item\label{item-Hnu-d}
		Denoting by $\mathcal{D}_\nu= \mathcal{H}_\nu- \textbf{H}_\nu$, for any $0<\delta<1$ there exists $C_\delta$ such that
		\[
		|\mathcal{D}_\nu(x+h,y)-\mathcal{D}_\nu(x,y)| \leq C_\delta\left( \frac {|x-y|}{\rho(x)} \right) \frac{|h|^\delta}{|x-y|^{d-2\nu +1+\delta}},
		\]
		
		as long as $|h|\leq  |x-y|/2 \leq \rho(x)$.
	\end{enumerate}
	\begin{proof}
		We follow the same steps as in Lemma~\ref{lem-Jgamma}, starting now with the identity:
		\[
		\nabla L^{-\nu}=-\frac{1}{2\pi} \int_{-\infty} ^\infty (-i\tau)^{-\nu} \nabla(L+i\tau)^{-1}\,\, d\tau.
		\]
		Hence
		\[
		\mathcal{H}_\nu (x,y)= -\frac{1}{2\pi} \int_{-\infty} ^\infty (-i\tau)^{-\nu} \nabla \Gamma(x,y,\tau)\,\, d\tau.
		\]
		Taking absolute value, using the estimate given in the proof of Lemma~\ref{lem:Gamma}, part~\eqref{item:lemgamma1} for $\nabla\Gamma$, changing variables and computing the integral we easily obtain~\eqref{item-Hnu-a}.
		As for~\eqref{item-Hnu-b} we use again the same expression for $\mathcal{H}_\nu $, using now item~\eqref{item:suavgradgammatau} from Lemma~\ref{lem:Gamma}.
		To check~\eqref{item-Hnu-c} we write the corresponding identity 
		\[
		\textbf{H}_\nu (x,y)= -\frac{1}{2\pi} \int_{-\infty} ^\infty (-i\tau)^{-\nu} \nabla \Gamma_0(x,y,\tau)\,\, d\tau.
		\]
		Subtracting both, using part~\eqref{tamagradgammatau-gammatau0} of Lemma~\ref{lem:Gamma} and following the same steps we arrive to the stated estimate.
		
		Finally,~\eqref{item-Hnu-d} follows using again the two identities and item~\eqref{item:suavgraddifdedif} in Lemma~\ref{lem:Gamma}.
	\end{proof}
\end{lem}

Applying the estimates proved above we may obtain the following result.

\begin{prop}\label{prop-SgammaesSCZ}
	Let $V\in RH_{d/2}$ and  $0<\gamma\leq 1/2$. If for some $\alpha>0$, $V$ satisfies~\eqref{eq-suavVgamma} then $S_\gamma$ is a  Schr\"odinger-Calder\'on-Zygmund operator of type~$(\infty,\alpha)$.
\end{prop}

\begin{proof}
	
	First, notice that $S_\gamma$ is of weak type $(1,1)$ as a consequence of Theorem~1 and Theorem~3 from~\cite{bongioanni2019weighted}.
		
	Secondly, our assumption on $V^\gamma$ implies $V(x) \leq C/\rho^2(x)$ and hence we may apply Lemma~\ref{lem-Hnu} for $\nu= \gamma+1/2$. In this way, for the kernel of $S_\gamma$ we have
	\begin{equation}\label{eq-auxS1}
	\begin{split}
	|V^\gamma(x) \mathcal{H}_{\gamma+1/2}(x,y)| & \leq \frac{C_N}{\rho^{2\gamma}(x)|x-y|^{d-2\gamma}} \left( 1+ \frac{|x-y|}{\rho(x)}\right) ^{-N} 
	\\& \leq C_N \frac{1}{|x-y|^d} \left( 1+ \frac{|x-y|}{\rho(x)}\right) ^{2\gamma-N} 
	\end{split}  
	\end{equation}
	Taking $N$ big enough we obtain the right size for the kernel. Regarding the smoothness we proceed as in the case of $T_\gamma$ to obtain for $\tilde{N}=2\gamma(1+N_0) +\alpha$
	
	\begin{equation}
	\begin{split}
	 |V^\gamma(x) &\mathcal{H}_{\gamma+1/2}(x,y) - V^\gamma(z) \mathcal{H}_{\gamma+1/2}(z,y)|
	  \\ & \leq V^\gamma(x)|\mathcal{H}_{\gamma+1/2}(x,y)
	  -\mathcal{H}_{\gamma+1/2}(z,y)| 
	  + |V^\gamma(x) - V^\gamma(z)| |\mathcal{H}_{\gamma+1/2}(z,y)|\\
	&  \leq C_N  \frac{|x-z|}{|x-y|^{d+1}}\left( 1+ \frac{|x-y|}{\rho(x)}\right) ^{2\gamma-N} \!\!\!\!  + C_N  \frac{|x-z|^\alpha}{|x-y|^{d+\alpha}}\left( 1+ \frac{|x-y|}{\rho(x)}\right) ^{\tilde{N}-N},
	\end{split}
	\end{equation}
	where we applied parts~\eqref{item-Hnu-a} and~\eqref{item-Hnu-b} from Lemma~\ref{lem-Hnu} and assumption~\eqref{eq-suavVgamma} on the potential  $V$, which imply~\eqref{eq-tamV} and~\eqref{eq-difVgamma}. Since $|x-z| \leq  |x-y|/2$, taking $N$ big enough, we get smoothness of order $\alpha$.
\end{proof}

Now we are in position to prove Theorem~\ref{thm:SGamma}.

\begin{proof}[Proof of Theorem~\ref{thm:SGamma}]

	In order to apply Theorem~\ref{teo-T1}, it remains to prove the $T1$ condition~\eqref{condT1}. For that we observe that the first inequality in~\eqref{eq-auxS1} shows that the kernel of $S_\gamma$ is integrable and hence $S_\gamma 1$ is finite everywhere. Now, we follow exactly the same steps as in the case of $T_\gamma$, replacing $\JS_\gamma$ by $\mathcal{H}_{\gamma+1/2}$, using this time estimates~\eqref{item-Hnu-a},~\eqref{item-Hnu-c} and~\eqref{item-Hnu-d} with $\nu=\gamma+1/2$ provided in Lemma~\ref{lem-Hnu}. In this way we also get
	\[
	|S_\gamma 1 (x)- S_\gamma 1(z)| \leq C \left( \frac{|x-z|}{\rho(x)}\right) ^\alpha,
	\]
	provided $|x-z|\leq \rho(x)/2$.
	
\end{proof}

\section{Singular Schrödinger Riesz transforms}\label{sec-reg-sing}

In this section we prove the  regularity results for the Schrödinger-Riesz transforms $\mathcal{R}_1$ and $\mathcal{R}_2$ stated in Theorem~\ref{thm:R1nuevo} and Theorem~\ref{teo:R2BMO}. The first case is almost straightforward while the second will require new and more refined estimates.

\subsection{First order singular Schr\"odinger Riesz transforms.} 
Before proving Theorem~\ref{thm:R1nuevo} we make a discussion about the already known results. As it was mentioned in Section~\ref{sec-intro},  boundedness results on $BMO_\rho^\beta (w)$ for $\RS_1$ can be found in Theorem~1 of~\cite{BHS-RieszJMAA} (see also~\cite{MR3180934_stinga}). The precise statement is the following.

\begin{thm}\label{thm:R1viejo}
	Let $V\in RH_q$ for $q> d$ and $\delta=1-d/q$. Then, for any 
	$0\leq\be<\de$, $R_1$ is bounded on $BMO_{\rho}^{\beta}(w)$ as long as $w\in A_{\infty}^{\rho}\cap D_{\mu}^{\rho}$, with $\displaystyle 1\leq \mu < 1+ \frac{\de-\be}{d}$. 
\end{thm}
Let us observe that boundedness on the whole range $0\leq\beta<1$, may be obtained asking $V\in \cap_{q>1}RH_q$ and that Theorem~\ref{thm:R1nuevo} give us the same conclusion but with different assumptions, namely ,$V\in RH_{d/2}$ and $V(x)\leq C/\rho^2(x)$.	

Even these latter conditions, as we showed, are satisfied for any $V\in RH_\infty$, that is no longer true for $V\in \cap_{q>1}RH_q$. In fact, $V(x)= \max\lbrace 1, \log\frac{1}{|x|}\rbrace$ belongs to $RH_q$ for all $q$ but inequality \eqref{eq-tamV} can not hold since $V$ is unbounded near zero. On the other hand, the potential $V(x)= \frac{1}{\left( 1+|x|\right) ^{2-\varepsilon}}$ satisfies the new conditions but it is not in $RH_q$ for $q$ large enough depending on $\varepsilon$. Clearly all $V$ in $RH_\infty$ satisfy both sets of assumptions.

In the sequel, we outline the proof of Theorem~\ref{thm:R1nuevo} following the same techniques developed for the previous cases.

\begin{proof}[Proof of Theorem~\ref{thm:R1nuevo}.]
	Again, the result will be a consequence of Theorem~\ref{teo-T1}. 
	 The weak type~$(1,1)$ follows from Theorem 7.3 in~\cite{BCH3_MR3541813}. To check the other conditions required to be a Schrödinger-Calderón-Zygmund operator of order $(\infty, \delta)$ for any $0<\delta<1$ we make use of Lemma~\ref{lem-Hnu} with  $\nu=1/2$.
	
	 Finally we need to show that $\RS_1$ satisfies the $T1$ condition~\eqref{condT1}. Using the estimates on the difference with the kernel of the classical Riesz transform given in Lemma~\ref{lem-Hnu} for $\nu=1/2$ we can go over  the same steps as in the proof of Theorem~\ref{thm:SGamma}. Actually, we have to handle with a principal value since $\RS_1$ is a singular integral operator. For the details we refer the reader to the proof of Theorem~\ref{teo:R2BMO} below.
\end{proof}

\subsection{Second order singular Schr\"odinger Riesz transforms}
The case $\RS_2$ is more difficult since we do not have the needed estimates at our disposal and they cannot be obtained from Lemma~\ref{lem:Gamma} as before. Consequently we divide the proof in several steps. First we prove that $\RS_2$ satisfies the requirements to be a Schr\"odinger-Calder\'on-Zygmund operator.

\begin{prop}\label{prop-R2SCZO}
	Let $V\in RH_{d/2}$ and suppose that $V$ satisfies~\eqref{eq-suavV} for some $0<\alpha\leq1$. Then $\RS_2$ is a Schr\"odinger Calder\'on Zygmund operator of type~$(\infty,\alpha)$.
\end{prop}

\begin{proof}
	The weak type~$(1,1)$ for $\RS_2$ is shown in Theorem~9 of~\cite{bongioanni2019weighted}. In order to prove estimates~\eqref{TamPuntual} and~\eqref{suav-puntual} we are going to derive a local expression for the kernel $\KS_2$ of $\RS_2$. Let $x$ and $y\in \RR^d$ and $R=|x-y|$. Let $x_0$ such that $|x_0-x|<R/8$. Consider a function $\eta_{x_0}\in \mathcal{C}_0^\infty(B(x_0,R/2))$ such that $\eta_{x_0}\equiv 1$ in $B(x_0,R/4)$, $|\nabla \eta_{x_0}|\leq C/R$ and $|\nabla^2 \eta_{x_0}|\leq C/R^2$. First we write,
	\begin{equation*}
	\begin{split}
	\Delta (\Gamma(\cdot,y)\eta_{x_0}) & = \eta_{x_0}V\Gamma(\cdot,y) + \Delta \eta_{x_0}\Gamma(\cdot,y) + 2\nabla \eta_{x_0} \cdot \nabla  \Gamma(\cdot,y)\\
	&=g_1+g_2+g_3=g.
	\end{split}
	\end{equation*}
	Then, if $z\in B(x_0,R/4)$,
	\begin{equation*}
\Gamma(z,y)=\eta_{x_0}(z)\Gamma(z,y)=\int \Gamma_0(z-\xi)g(\xi)d\xi.
	\end{equation*}
	Therefore,
	\begin{equation*}
	\begin{split}
	\KS_2(z,y)=\nabla^2( \Gamma(\cdot,y) )(z)  & = \nabla^2 \left(\int 
	\Gamma_0(z-\xi)g(\xi)d\xi\right)
	\\ & 
	= \nabla^2\int \Gamma_0(z-\xi)g_1(\xi)d\xi + \int \nabla^2\Gamma_0(z-\xi)g_2(\xi)d\xi
	\\ & \,\,\,\,\,\,\,+\int \nabla^2\Gamma_0(z-\xi)g_3(\xi)d\xi\\ & = \KS_{2,1}(z,y) + \KS_{2,2}(z,y) + \KS_{2,3}(z,y),
	\end{split}
	\end{equation*}
	due to the fact that the second and third integrals above are absolutely convergent since they are supported on a bounded domain away from the diagonal. For the first term we are going to prove that $g_1$ is a Lipschitz function of order $\alpha$ supported on $B(x_0,R/2)$. In that case, it is known that we have the following representation for $z\in B(x_0,R/4)$ (see Lemma 4.4 in~\cite{gilbarg2001elliptic}), 
	\begin{equation}\label{eq-K21conz}
	\begin{split}
	\KS_{2,1}(z,y)  & =  \int_{B(x_0,R/2)} \nabla^2\Gamma_0(z-\xi)[g_1(\xi)-g_1(z)]d\xi + c_d\mathbf{I}(g_1)(z)
	\\ & = \RC_2(g_1)(z) + c_d\mathbf{I}(g_1)(z),
	\end{split}
	\end{equation}
	where $\RC_2=\nabla^2(-\Delta)^{-1}$ is the classical Riesz transform of second order, $(\mathbf{I}(g_1))_{i,j}=g_1\delta_{i,j}$  and moreover the integral is absolutely convergent. In particular, taking $z=x$, we have that
	\begin{equation}\label{eq-estK21}
	\begin{split}
	|\KS_{2,1}(x,y)| & \leq C \|g_1\|_{\textup{Lip}^{\alpha}}\int_{B(x_0,R/2)}|\xi-x|^{\alpha-d}+c |g_1(x)|
	\\ & \leq C (\|g_1\|_{\textup{Lip}^{\alpha}} R^\alpha + \|g_1\|_{L^{\infty}}). 
	\end{split}
	\end{equation}
	
	We claim that not only $g_1$ is a Lipschitz-$\alpha$ function, but its norm is bounded by
	
	\begin{equation}\label{eq-NormaLipg1}
	\|g_1\|_{\textup{Lip}^{\alpha}}\leq \frac{C_N}{R^{d+\alpha}} \left(\frac{R}{\rho(y)}\right)^2\left(1+\frac{R}{\rho(y)}\right)^{-N}.
	\end{equation}
	Assuming this is true and using that $g_1=0$ on $B(x_0,R/4)^c$,
	\begin{equation*}
	\|g_1\|_{L^{\infty}} \leq C R^\alpha \|g_1\|_{\textup{Lip}^{\alpha}}
	\end{equation*}
	This, together with~\eqref{eq-estK21} and~\eqref{eq-NormaLipg1} give us
	\begin{equation}\label{eq-tamK21}
	|\KS_{2,1}(x,y)|\leq \frac{C_N}{R^d}\left(\frac{R}{\rho(y)}\right)^2\left(1+\frac{R}{\rho(y)}\right)^{-N},
	\end{equation}
	which implies 
	\begin{equation*}
	|\KS_{2,1}(x,y)|\leq \frac{C_N}{|x-y|^d}\left(1+\frac{|x-y|}{\rho(y)}\right)^{-N+2},
	\end{equation*}
	as we wanted to show. 
	
	For the smoothness estimate we take $x'$ such that $|x-x'|<R/8$. Then $x'\in B(x_0,R/4)$ and formula~\eqref{eq-K21conz} applies. Since  $\RC_2$ is a continuous operator on $\textup{Lip}^\alpha$, applying~\eqref{eq-NormaLipg1},
	\begin{equation}\label{eq-suavK21}
	\begin{split}
	|\KS_{2,1}(x,y)-\KS_{2,1}(x',y)| & \leq C |x-x'|^\alpha \|g_1\|_{\textup{Lip}^{\alpha}}
	\\ & \leq C_N \frac{|x-x'|^{\alpha}}{R^{d+\alpha}} \left(\frac{R}{\rho(y)}\right)^2\left(1+\frac{R}{\rho(y)}\right)^{-N}.
	\end{split}
	\end{equation}
	In particular,
	\begin{equation*}
	|\KS_{2,1}(x,y)-\KS_{2,1}(x',y)| \leq C \frac{|x-x'|^{\alpha}}{|x-y|^{d+\alpha}},
	\end{equation*}
	as long as $|x-x'|<|x-y|/8$.
	
		It remains to prove the claim. First, recall that $\text{supp}\,\eta_{x_0}\subset B(x_0,R/2)\subset B(x,5R/8)$. So, if $z\notin B(x,5R/8)$ we have that $V(z)\Gamma(z,y)\eta_{x_0}(z)= 0$. If $z\in B(x,5R/8)$, using that $|z-y|\simeq R$ and~\eqref{eq-tamV},
	\begin{equation*}
	\begin{split}
	|V(z)\Gamma(z,y)\eta_{x_0}(z)| & \leq \frac{C_N}{R^{d}} \left(\frac{R}{\rho(z)}\right)^{2} \left(1+\frac{R}{\rho(z)}\right)^{-N} \\ &=\frac{C_N R^\al}{R^{d+\al}} \left(\frac{R}{\rho(y)}\right)^{2} \left(1+\frac{R}{\rho(y)}\right)^{-N/(N_0+1)+2N_0}.
	\end{split}
	\end{equation*} 
	Then, if $|z-z'|\geq R/4$,
	\begin{equation}
	|V(z)\Gamma(z,y)\eta_{x_0}(z) - V(z')\Gamma(z',y)\eta_{x_0}(z')| \leq C \frac{|z-z'|^{\alpha}}{R^{d+\alpha}}\left(\frac{R}{\rho(y)}\right)^{2} \left(1+\frac{R}{\rho(y)}\right)^{-\tilde{N}},
	\end{equation}
	with $\tilde{N}=N/(N_0+1)+2N_0$.
	
	On the other hand, if $|z-z'|<R/4$, we can write
	\begin{equation*}
	\begin{split}
	|V(z)\Gamma(z,y)\eta_{x_0}(z) - V(z')\Gamma(z',y)\eta_{x_0}(z')| 
	& \leq V(z) \eta_{x_0}(z) |\Gamma(y,z)-\Gamma(y,z')| 
	\\ & +
	V(z) |\Gamma(z',y)| |\eta_{x_0}(z)-\eta_{x_0}(z')| 
	\\ & + |\Gamma(z',y)| \eta_{x_0}(z') |V(z)-V(z')|
	\\ & = I + II + III.
	\end{split}
	\end{equation*}
	
	In $A$ we have that if $z\notin B(x,5R/8)$, then  $I=0$. If $z\in B(x,5R/8)$, then $z'\in B(x,7R/8)$ and $|z-y|\simeq |z'-y|\simeq R$. So, applying~\ref{eq-tamV}, Lemma~\ref{lem:Gamma}, and inequality~\eqref{rox_vs_roy},
	\begin{equation*}
	\begin{split}
	I & \leq \frac{C_N|z-z'|}{R^{d+1}} \left(\frac{R}{\rho(z)}\right)^2 \left(1+ \frac{R}{\rho(y)}\right)^{-{N}} \leq \frac{C_N|z-z'|}{R^{d+1}} \left(\frac{R}{\rho(y)}\right)^2 \left(1+ \frac{R}{\rho(y)}\right)^{-\tilde{N}},
	\end{split}
	\end{equation*}
	with $\tilde{N}=N+2N_0$.
	
	If $z$, $z'\notin B(x,5R/8)$ we have that $II=0$. If instead, $z\in B(x,5R/8)$, it follows that $z'\in B(x,7R/8)$ y $|z-y|\simeq |z'-y|\simeq R$. Then, applying again ~\ref{eq-tamV}, Lemma~\ref{lem:Gamma} and inequality~\eqref{rox_vs_roy},
	\begin{equation*}
	II \leq C_N\frac{|z-z'|}{R^{d+1}} \left(\frac{R}{\rho(z)}\right)^2 \left(1+\frac{R}{\rho(y)}\right)^{-N} \leq \frac{C_N|z-z'|}{R^{d+1}} \left(\frac{R}{\rho(y)}\right)^2 \left(1+ \frac{R}{\rho(y)}\right)^{-\tilde{N}},
	\end{equation*}
	with $\tilde{N}=N+2N_0$.

	To estimate $III$ we may also observe that if $z'\notin B(x,5R/8)$,  then $III=0$. On the other hand, if $z'\in B(x,5R/8)$, then $z\in B(x,7R/8)$ and $|z-y|\simeq |z'-y|\simeq R$. Therefore, if $|z-z'|<\rho(z)$ we may use~\eqref{eq-suavV}and inequality~\eqref{rox_vs_roy} to obtain
	\begin{equation*}
	\begin{split}
	III & \leq C_N \frac{|z-z'|^\alpha}{\rho(z)^{d+\alpha}R^{d-2}}  \left(1+\frac{R}{\rho(y)}\right)^{-N}
	\\& \leq C_N \frac{|z-z'|^\alpha}{R^{d+\alpha}} \left(\frac{R}{\rho(y)}\right)^{2} \left(1+\frac{R}{\rho(y)}\right)^{-\tilde{N}},
	\end{split}
	\end{equation*}
	with $\tilde{N}=-N+(2+\alpha)N_0+\alpha$. If $|z-z'|\geq \rho(z)$, we may write $|V(z)-V(z')|\leq V(z) + V(z')$ obtaining two terms $III_1$ and $III_2$. In each of those terms we can apply~\eqref{eq-tamV} and inequality~\eqref{rox_vs_roy} to obtain
	\begin{equation*}
	\begin{split}
	III_1\ &\leq  \frac{C_N}{\rho^2(z) R^{d-2}} \left(1+\frac{R}{\rho(y)}\right)^{-N}
	\\ &\leq  \frac{C_N}{\rho^2(z) R^{d-2}} \left(1+\frac{R}{\rho(y)}\right)^{-N} \left(\frac{|z-z'|}{\rho(z)}\right)^\alpha
		\\ &\leq  \frac{C_N|z-z'|^\alpha}{\rho^{2+\alpha}(y) R^{d-2}} \left(1+\frac{R}{\rho(y)}\right)^{-N+(2+\alpha)N_0}
	\\& \leq C_N \frac{|z-z'|^\alpha}{R^{d+\alpha}} \left(\frac{R}{\rho(y)}\right)^{2} \left(1+\frac{R}{\rho(y)}\right)^{-N+(2+\alpha)N_0+\alpha},
	\end{split}
	\end{equation*}
	 A similar estimate can be obtained for $III_2$. The estimates for $I$, $II$ and $III$ together give us the claimed bound for $\|V \Gamma(\cdot,y) \eta_{x_0} \|_{\text{Lip}^{\al}}$.

	Now, we turn our attention to $\KS_{2,2}$. If $z\in B(x_0,R/4)$,
	\begin{equation*}
	\KS_{2,2}(z,y)=\int \nabla_1^2\Gamma_0(z-\xi)\Delta\eta_{x_0} \Gamma (\xi,y)d\xi.
	\end{equation*}
	In particular, if $z=x$,
	\begin{equation*}
	|\KS_{2,2}(x,y)| \leq \frac{C_N}{R^2}\int_{R/4<|\xi-x_0|<R/2} \frac{1}{|x-\xi|^d|y-\xi|^{d-2}}\left(1+\frac{|y-\xi|}{\rho(y)}\right)^{-N}d\xi.
	\end{equation*}
	In this situation it can be shown that $|x-\xi|\simeq|y-\xi|\simeq R$. Then,
	\begin{equation*}
	|\KS_{2,2}(x,y)| \leq \frac{C_N}{|x-y|^d} \left(1+\frac{|x-y|}{\rho(x)}\right)^{-N}.
	\end{equation*}
	
	In a similar way, if $|x-x'|\leq |x-y|/8$,
	\begin{equation*}
	\begin{split}
	|\KS_{2,2}(x,y)-\KS_{2,2}(x',y)| & \leq \int \left|\nabla_1^2\Gamma_0(x-\xi) -\nabla_1^2\Gamma_0(x'-\xi)\right| 
	|\Delta\eta_{x_0} |\Gamma (\xi,y)d\xi
	\\ & \leq \frac{C|x-x'|}{R^2}\int_{R/4<|\xi-x_0|<R/2} \frac{1}{|x-\xi|^{d+1}|y-\xi|^{d-2}}d\xi
	\\ & \leq \frac{C|x-x'|}{|x-y|^{d+1}},
	\end{split}
	\end{equation*}
	since again $|x-\xi|\simeq|y-\xi|\simeq R$.
	
	Finally we need to obtain the corresponding estimates for $\KS_{2,3}$. If $z\in B(x_0,R/4)$
	\begin{equation*}
	\KS_{2,3}(z,y)=\int 2\nabla_1^2\Gamma_0(z-\xi)\nabla\eta_{x_0} \cdot \nabla_1\Gamma (\xi,y)d\xi.
	\end{equation*}
	In particular, if $z=x$,
	\begin{equation*}
	|\KS_{2,3}(x,y)| \leq \frac{C_N}{R}\int_{R/4<|\xi-x_0|<R/2} \frac{1}{|x-\xi|^d|y-\xi|^{d-1}}\left(1+\frac{|y-\xi|}{\rho(y)}\right)^{-N}d\xi.
	\end{equation*}
	As before, it can be shown that $|x-\xi|\simeq|y-\xi|\simeq R$. Then,
	\begin{equation*}
	|\KS_{2,3}(x,y)| \leq \frac{C_N}{|x-y|^d} \left(1+\frac{|x-y|}{\rho(x)}\right)^{-N}.
	\end{equation*}
	
	In a similar way, if $|x-x'|\leq |x-y|/8$,
	\begin{equation*}
	\begin{split}
	|\KS_{2,3}(x,y)-\KS_{2,3}(x',y)| & \leq \int \left|\nabla_1^2\Gamma_0(x-\xi) -\nabla_1^2\Gamma_0(x'-\xi)\right| 
	|\nabla\eta_{x_0} ||\nabla\Gamma (\xi,y)d\xi|
	\\ & \leq \frac{C|x-x'|}{|x-y|^{d+1}},
	\end{split}
	\end{equation*}
	since again $|x-\xi|\simeq|y-\xi|\simeq R$.

\end{proof}

Next we obtain some estimates involving the difference between the kernels of $\RS_2$ and the corresponding to the classical second order Riesz transform $\RC_2$, denoted by $\KC_2$.
\begin{lem}\label{lem-comparacionR2}
	Let $V\in RH_{d/2}$ and such that inequality  \eqref{eq-suavVgamma} holds for some $\alpha \leq 1$. Then, for any $y$, $x \in\RR^d$  such that $|y-x|\leq \rho(x)$ there exists a constant $C$ such that
	\begin{enumerate}[(a)]
		\item\label{item-DifR2-a}
		\begin{equation*}
		|\KS_2(x,y)-\KC_{2}(x,y)|\leq
		C  \frac{C}{|x-y|^d}\left(\frac{|x-y|}{\rho(x)}\right)^{2}.
		\end{equation*}
		\item\label{item-DifR2-b} For any $x'$ such that $|x-x'|< \frac{1}{16}|x-y|$, we have
		\[
		|\KS_2(x,y)-\KC_2(x,y)-\left[\KS_2(x',y)-\KC_2(x',y)\right]|
		\leq C \frac{|x-x'|^\alpha}{|x-y|^{d+\alpha}} \left( \frac{|x-y|}{\rho(x)} \right) ^2.
		\]
	\end{enumerate}
\end{lem}

\begin{proof}
	Let $\Gamma$ and $\Gamma_0$ be the fundamental solutions of $L$ and $-\Delta$ respectively. We remind that 	
	\begin{equation}
	\Gamma(x,y)- \Gamma_0(x,y) = -\int_{\RR^d}\Gamma_0(x,\xi)V(\xi)\Gamma(y,\xi) d\xi.
	\end{equation} 
	From this identity it follows
	\begin{equation}
	\KS_2(x,y)- \KC_{2}(x,y) = \nabla_1^2\Gamma(x,y)- \nabla_1^2\Gamma_0(x,y) = -\nabla_1^2\int_{\RR^d}\Gamma_0(x,\xi)V(\xi)\Gamma(y,\xi) d\xi.
	\end{equation} 
	As in Lemma~\ref{lem-Hnu}, let $ R=|x-y|$ and $x_0$ such that $|x-x_0|< R/8$ and $\eta_{x_0}$ a $C_0^\infty$ function supported on $B(x_0,R/2)$ and equals one on $B(x_0,R/4)$. Set $D(x,y)= \KS_2(x,y)-\KC_{2}(x,y)$.
	
	Then we write
	\begin{equation}
	\begin{split}
D(x,y)&= -\nabla^2_1\int_{\RR^d}\Gamma_0(x,\xi)V \Gamma(y,\xi) \eta_{x_0}(\xi)d\xi \\
& \hspace{50pt} - \int_{\RR^d} \nabla^2_1 \Gamma_0(x,\xi)V \Gamma(y,\xi) (1-\eta_{x_0})(z)d\xi
\\&=D_1(x,y)+D_2(x,y).
	\end{split}
	\end{equation}
	So, it is enough to prove size and smoothness for each term. In fact, as we shall see, the second term has a better regularity since inequality~\eqref{item-DifR2-b} holds with $\alpha=1$. 
	
	For $D_1$ we notice that it is exactly the same as $\mathcal{K}_{21}$ appearing in the proof of Proposition~\ref{prop-R2SCZO} and hence inequalities~\eqref{eq-tamK21} and~\eqref{eq-suavK21} give the right hand side of~\eqref{item-DifR2-a} and~\eqref{item-DifR2-b} respectively.

	Now for $ D_2$ we observe that $1-\eta_{x_0}$ is supported on $B^c(x_0, R/4)$ and that it is possible to take derivatives inside the integral because, as we shall see later, they converge absolutely.
	We split the integral over the subsets $J_1=B(y,R/2)$ 
	and $J_2= B^c(x_0, R/4)\cup J^c_1$.
	
	First, on $J_1$ we are away from the singularity of $\Gamma_0$, more precisely $|x-\xi| \geq |x-y|-|y-\xi| \geq 3R/4$, so we may use the size estimates for $\Gamma$ and $\Gamma_0$ (see Lemma~\ref{lem:Gamma}) together with~\eqref{eq-tamV} to obtain
	\begin{equation}
	\begin{split}
	\int_{J_1}\left|\nabla_1^2\Gamma_0(x,\xi)\right|V(\xi)\Gamma(y,\xi)d\xi
	& \leq
	\frac{C}{R^d} \int_{B(y,R/4)} \frac{V(\xi)}{|y-\xi|^{d-2}}d\xi
	\\ & \leq
	\frac{C}{R^{d}} \left(\frac{R}{\rho(x)}\right)^2, 
	\end{split}
	\end{equation}
	where we also use $\rho(\xi)\simeq \rho(x)$	since $|\xi-x|\leq 2\rho(x)$.
	
	Similarly, for the smoothness of this term, taking $|x-x'|<R/16$,  we have $|\xi-x|\geq R/2$.                                                                                                                                                                                           Therefore  $|x-x'|<\frac{1}{8}| \xi-x|$ and we may apply the smoothness property of $\nabla_1^2\Gamma_0$ away from the diagonal to get	
	\begin{equation}
	\begin{split}
	\int_{J_1}\left|\nabla_1^2\right.&\Gamma_0(x,\xi)- \nabla_1^2 \Gamma_0(x',\xi)| V(\xi)\Gamma(y,\xi)d\xi
	\\ & \leq
	\frac{C |x-x'|}{R^{d+1}} \int_{B(y,R/4)} \frac{V(\xi)}{|y-\xi|^{d-2}}d\xi
	\\ & \leq
	C \frac{|x-x'|}{R^{d+1}} \left(\frac{R}{\rho(x)}\right)^2, 
	\end{split}
	\end{equation}

	To estimate the integral over $J_2$ we write $ J_2=J_{21}\cup J_{22}$ where
	$J_{21}=\{\xi \in \RR^d: R/4\leq|y-\xi|<2R \,\land\, |x_0-\xi|\geq R/4\}$ and  $J_{22}=\{\xi \in \RR^d:|y-\xi|\geq 2R\}.$
	
	On $J_{21}$ we are away from the singularities of $\Gamma$ and $\Gamma_0$, then
	\begin{equation}
	\begin{split}
	\left|\int_{J_{21}}\nabla_1^2\Gamma_0(x,\xi)V(\xi)\Gamma(y,\xi)d\xi\right|
	& \leq
	\frac{C}{R^d} \left(\frac{R}{\rho(x_0)}\right)^2,
	\end{split}
	\end{equation}
	where we used again the bound on $V$ and that $\rho(\xi)\simeq \rho(x)$.
	
	Regarding the smoothness, since $|\xi-x|\geq |\xi-x_0|- |x_0-x|\geq R/8 $ and we are assuming $|x-x'|<R/16$, we can apply again the regularity of $\nabla^2 \Gamma_0$ to also obtain in this case the right hand side of~\eqref{item-DifR2-b}.

	When integrating on $J_{22}$, it is easy to see that $|x-\xi|\geq 3|y-\xi|/8 \geq R/4$. Therefore, using this time the extra decay of $\Gamma$, 
	\begin{equation}
	\begin{split}
	\left|\int_{J_{22}}\nabla_1^2\Gamma_0(x,\xi)V(\xi)\Gamma(y,\xi)d\xi\right|
	& \leq C
	\int_{J_{22}} \frac{V(\xi)}{ |y-\xi|^{2d-2}} \left( 1+\frac{|\xi - y|}{\rho(y)}\right) ^{-N} d\xi  
	\\ & \leq
	\frac{C}{R^d} \left(\frac{R}{\rho(x_0)}\right)^2,
	\end{split}
	\end{equation}
	where we made use of the inequality
	\[
	\frac{1}{\rho^2(\xi)}\left( 1+\frac{|\xi - y|}{\rho(y)}\right) ^{-N} \leq \frac{C}{\rho^2(x)}  
	\]
	that holds for any $N>2N_0$.
	
	
	Finally, for the smoothness of this term, we check again that we are away from the diagonal. In fact, if $|x-x'|< R/16$, since $|x-\xi| \geq |\xi-y|-|y-x| \geq 2R-R=R$ we get that $|x-x'|< |\xi-x|/16$. Now, proceeding as above, we arrive to the desired inequality.
\end{proof}

Now we give a proof of Theorem~\ref{teo:R2BMO}. It follows the ideas of the poofs given for Theorem~\ref{thm:TGamma} and Theorem~\ref{thm:SGamma}, although some changes must be done since $\RS_2$ is a singular integral operator.

\begin{proof}[Proof of Theorem~\ref{teo:R2BMO}]
	We use Proposition~\ref{prop-R2SCZO} together with Theorem~\ref{teo-T1}. Therefore, it is enough to show that $\RS_2(1)$ satisfies condition~\eqref{condT1}. First, notice that $\RS_2 1$ is finite everywhere as a consequence of Lemma~\ref{lem-comparacionR2} part~\eqref{item-DifR2-a} and the size estimate for $\KS_2$ given in Proposition~\ref{prop-R2SCZO}.
	
	Also, we make the following observation: for $x$ and $z\in \RR^d$
	\begin{equation}
	\begin{split}
	\RS_2(1)(x) &
	=\lim_{\varepsilon\rightarrow0} \int_{\RR^d\setminus B(x,\varepsilon)}\KS_2(x,y)dy
	=\lim_{\varepsilon\rightarrow0} \int_{\RR^d\setminus (B(x,\varepsilon)\cup B(z,\varepsilon))}\KS_2(x,y)dy,
	\end{split}
	\end{equation}
	since the integral over $B(z,\varepsilon)$ goes to zero with $\varepsilon$.

	Now, $x$, $z\in\RR^d$ such that $|x-z|\leq \rho(x)/2$. Setting $A_\varepsilon=B(x,\varepsilon)\cup B(z,\varepsilon)$, $\KC_2$ the kernel of the classical Riesz transform and $D(x,y)=\KS_2(x,y)-\KC_2(x-y)$ we may write
	
	\begin{equation*}
	\begin{split}
	\RS_2(1)(x) - \RS_2(1)(z) & = 
	\lim_{\varepsilon\rightarrow 0^+} \int_{\RR^d\setminus A_\varepsilon} \left[\KS_2(x,y) dy - 
	\KS_2(z,y)\right] dy
	\\ & = \lim_{\varepsilon\rightarrow 0^+} \int_{B(x,\rho(x_))\setminus A_\varepsilon} 
	 \left[D(x,y)-D(z,y)\right]dy 
	\\ & \hspace{30pt} + \lim_{\varepsilon\rightarrow 0^+} \int_{B(x,\rho(x))\setminus A_\varepsilon} 
	\left[\KC_2(x-y)-\KC_2(z-y)\right]dy
	\\ & \hspace{60pt} +  \int_{B(x,\rho(x))^c} 
	\left[\KS_2(x,y)-\KS_2(z,y)\right]dy\\
	& = I + II + III.
	\end{split}
	\end{equation*}
	For $I$ we may first write
	\begin{equation}
	\begin{split}
	|I| & \leq \int_{B(x,\rho(x))\setminus B(x,16|x-z|)} |D(x,y)-D(z,y)|dy
	\\ & \hspace{30pt} + \int_{B(x,16|x-z|)}|D(x,y)|dy \\
	& \hspace{60pt} + \int_{B(z,17|x-z|)}|D(z,y)|dy\\
	& = I_1 + I_2 + I_3.
	\end{split}
	\end{equation}

	To deal with $I_1$ we can apply part~\eqref{item-DifR2-b} of Lemma~\ref{lem-comparacionR2} to obtain
	\begin{equation*}
	\begin{split}
	I_1  & \leq \frac{C|x-z|^\alpha}{\rho^2(x)} \int_{|x-y|<\rho(x)} \frac{dy}{|x-y|^{d-2+\alpha}}
	\\ & \leq C\left(\frac{|x-z|}{\rho(x)}\right)^\alpha.
	\end{split}
	\end{equation*}
	
	As for $I_2$, we can apply part~\eqref{item-DifR2-a} of Lemma~\ref{lem-comparacionR2} to obtain
	\begin{equation*}
	\begin{split}
	I_2  & \leq  \frac{C}{\rho^2(x)}\int_{|x-y|<16|x-z|} \frac{dy}{|x-y|^{d-2}}
	\\ & \leq  C\left(\frac{|x-z|}{\rho(x)}\right)^2,
	\end{split}
	\end{equation*}
	since $|x-z|\leq\rho(x)$. We handle $I_3$ in the same way.

	Now we are going to estimate $II$. We may write
	\begin{equation*}
	\begin{split}
	|II| & \leq \left|\lim_{\varepsilon\rightarrow 0^+}\int_{B(x,\rho(x))\setminus B(x,\varepsilon)}\KC_2(x-y)dy\right|\\
	&\hspace{30pt} + \left|\lim_{\varepsilon\rightarrow 0^+}\int_{B(z,\rho(x))\setminus B(z,\varepsilon)}\KC_2(z-y)dy\right|
	\\ & \hspace{60pt} +\int_{B(x,\rho(x))\triangle B(z,\rho(x)) }|\KC_2(z-y)|\\
	& =II_1+II_2+II_3.
	\end{split}
	\end{equation*}
	Te terms $II_1$ and $II_2$ equal zero. As for $II_3$ we have $|z-y|\simeq \rho(x)$ and \[|B(x,\rho(x))\triangle B(z,\rho(x))|=C|x-z|\rho^{d-1}(x).\] Therefore,
	\begin{equation*}
	\begin{split}
	II_3 \leq C \rho^{-d}(x)|B(x,\rho(x))\triangle B(z,\rho(x))|
	\leq C \frac{|x-z|}{\rho(x)}.
	\end{split}
	\end{equation*}
	
	Finally, we turn our attention to $III$. Applying the smoothness condition for $\KS_2$ we may write
	\begin{equation*}
	\begin{split}
	|III|&\leq C|x-z|^\alpha\int_{B(x,\rho(x))^c }\frac{1}{|x-y|^{d+\alpha}} dy	  
	\leq C \left(\frac{|x-z|}{\rho(x)}\right)^\alpha.
	\end{split}
	\end{equation*}
	
\end{proof}

\section{A digression}\label{sec-digression}

All along this section we are going to assume that $w=1$. Our intention is to discuss to what extent the condition on $V$ required for boundedness of $T_{\gamma}$ on $BMO_\rho^\beta$ given in Theorem~\ref{thm:TGamma} is necessary.

We are able to show that, if $0<2\gamma<1$ and $\beta$ is small enough, a little weaker condition must hold. More precisely, if we assume $T_\gamma$ bounded on $BMO_\rho^\beta$, the potential must satisfy
\begin{equation}\label{eq-suavVgamma-2}
|V^\gamma(x)-V^\gamma(y)|\leq C\frac{|x-y|^{\beta}}{\rho^{2\gamma +\beta}(x)}, \ \ \ \text{for $|x-y|<\rho(x)$.}
\end{equation}
Remind that, according to Theorem~\ref{thm:TGamma}, to obtain boundedness of $T_{\gamma}$ on $BMO_\rho^\beta$  we need the above inequality to hold with some $\alpha> \beta$.
Nevertheless, for such small values of  $\beta$, and by means of a different technique, we  can get a more refined result, showing that the above condition is also sufficient.

Our first observation is that the operator $T_\gamma$ is the composition of a fractional integral with multiplication by a fixed function. Then, in order to prove that \eqref{eq-suavVgamma-2} is sufficient, we shall analyse the behaviour of $\L^{-\gamma}$ on $BMO_\rho ^\beta $. In~\cite{BHSnegativepowers} as well as in~\cite{MR3180934_stinga}, it is proved that under the assumption $V\in RH_q$ for some $q>d/2$, it maps $BMO_\rho^\beta$ into $BMO_\rho^{\beta +2\gamma}$ when $\beta+2\gamma < \delta_0$, where $\delta_0= \min \left\lbrace 1,2-d/q\right\rbrace $. However we will check that with the additional assumption $V(x) \leq C/\rho^2(x)$, the above boundedness holds for all $\beta$ such that $\beta+2\gamma <1$, that is like the parameter $q$ were $\infty$.

To do that we will apply Theorem 1.1 in~\cite{MR3180934_stinga}. Accordingly, we must check that it is bounded on some $L^p$, that its kernel satisfies certain size and smoothness  conditions and that the $T1$  holds with $2\gamma$ and any $\beta$ such that $2\gamma+ \beta<1$. The two first requirements are done in the proof of Theorem 1.4 in~\cite{MR3180934_stinga} (see pages 577 and 578). For the other two conditions we can not use their estimates since they would give us the restriction $2\gamma< \delta_0$.  However, under the extra hypothesis assumed, we obtained in Lemma~\ref{lem-Jgamma} that the kernel of $L^{-\gamma}$ satisfies smoothness with $\delta=1$ and, due to Remark~\ref{rem-T1L}, when $0< 2\gamma <1$, it satisfies
\[
|L^{-\gamma}1(x)-L^{-\gamma}1(y)|\leq C\left( \frac{|x-y|}{\rho(x)}\right)^{1-2\gamma} |x-y|^{2\gamma}
\]
as long as $|x-y|< \rho(x)$.

In this way we have proved:
\begin{prop}\label{continuity} 
	If $V\in RH_{d/2}$ and $V(x) \leq C/\rho^2(x)$, then for $0<2\gamma<1$, $L^{-\gamma}$ is bounded from $BMO_\rho^\beta$ into $BMO_\rho^{\beta +2\gamma}$ provided $\beta+2\gamma < 1$. 
\end{prop}
Our next lemma gives estimates for comparing the Schrödinger fractional integral with the classical one, which will be helpful for proving the necessity of condition  \eqref{eq-suavVgamma-2}. 

Following Shen, we call $\Gamma$ and $\Gamma_0$ to the corresponding fundamental solutions of $L+i \tau$ and $-\Delta +i\tau$. When $\tau=0$ they are the kernels of the operators $L^{-1}$ and $(-\Delta)^{-1}$. As before, for $0<\gamma<1$, we denote by  $\JS_\gamma$  and by $\JC_\gamma$ the kernels of $L^{-\gamma}$ and $\left( -\Delta\right) ^{-\gamma}$,  respectively.

\begin{lem} \label{compare}
	Let $V\in RH_{d/2}$. Then, there exists $\de_0>0$ such that
	\begin{enumerate}[(a)]
		\item\label{item-disg-a} For any $k>0$ there exists $C_k$
		such that
		\begin{equation}
		|\Gamma(x,y,\tau)-\Gamma_0(x,y,\tau)|\leq \left( \frac{|x-y|}{\rho(x)}\right) ^{\de_0} \frac{C_k}{\left( 1+|\tau|^{1/2}|x-y|\right)^k |x-y|^{d-2}},
		\end{equation}
		for $|x-y|< \rho(x)$. 
		\item\label{item-disg-b} For any $\gamma$ such that  $0<\gamma<1$,
		\begin{equation}
		|\JS_\gamma(x,y)-\JC_\gamma(x,y)| \leq \left( \frac{|x-y|}{\rho(x)}\right) ^{\de_0} \frac{C_\ga}{|x-y|^{d-2\gamma}},
		\end{equation}
		for $|x-y|< \rho(x)$.

	\end{enumerate}
	
	\begin{proof}
		Part~\eqref{item-disg-a} is proved in~\cite{shen} (see Lemma 4.5 in there).
		For~\eqref{item-disg-b} we just use the spectral formula, valid for $0<\gamma<1$
		\[
		L^{-\gamma}= -\frac{1}{2\pi} \int_{-\infty}^\infty (-i\tau)^{-\gamma} (L+i\tau)^{-1} d\tau
		\]
		as well as the same identity for $(-\Delta)^{-\gamma}$. Now, subtracting both identities, using part~\eqref{item-disg-a} and performing the integral, we easily arrive to the right hand side.

	\end{proof}
\end{lem}
Finally, before proving the equivalence result we provide a family of functions belonging to $BMO_\rho^\beta$ which will be also useful in showing the necessity of condition \eqref{eq-suavVgamma-2}. By the way, let us remind that for $0<\beta<1$, functions in $BMO_\rho^\beta$ have a point-wise description, that for $w=1$, according to~\eqref{cond_Lip_local} and~\eqref{cond_Lip_global}, are \begin{equation}\label{eq-averages}
|f(x)-f(y)|\leq C|x-y|^\beta
\end{equation} 
and
\begin{equation}\label{eq-smooth}
|f(x|\leq C \rho^\beta(x).
\end{equation} 

Moreover, for a function to be in $BMO_\rho^\beta$, $0<\beta<1$, it is enough to check~\eqref{eq-averages} for $x$, $y\in\RR^d$ such that $|x-y|< \rho(x)$ and~\eqref{eq-smooth}.

\begin{lem}\label{ejemplos}
	If for each point $x$ we consider the functions
	\begin{equation}
	f_{x}(z)= \max \lbrace \left( 2\rho(x)\right) ^{\beta } - |x-z|^{\beta },\,\,0\rbrace.
	\end{equation}
	Then, they belong to $BMO_\rho^\beta$ with uniformly bounded norms.
	
	\begin{proof}
		We fix $x$ and for simplicity we call just $f$ to the associated function.
		Clearly $|f(z)|\leq c \rho^\beta(z)$. In fact, if $z$ is such that $|x-z|\leq \rho(x)$ then $\rho(x) \simeq \rho(z)$ and otherwise $f(z)=0$. So, in any case, the constant $c$ is independent of the parameter $x$.
		Next, to check regularity, we observe that the function $h(t)= \max\lbrace t,0\rbrace$ is Lipschitz-1 and since $g(t)= t^{\beta}$ is Lipschitz-$\beta $, we have
		\begin{equation}
		|f(z)-f(z')|\leq ||x-z|^{\beta }- |x-z'|^{\beta }|\leq c|z-z'|^{\beta }
		\end{equation}
		with a constant independent of the parameter $x$ .
	\end{proof}
\end{lem}

Now we are in position to establish the equivalence result.
\begin{prop}
	Let $ V\in RH_{d/2}$ and assume that  $\alpha$ and $\beta$ are non-negative numbers such that $2\gamma + \beta <1$ Then, the operator $V^\gamma L^{-\gamma}$ is bounded on
	$BMO_\rho^\beta$ if and only if the potential $V$ satisfies condition \eqref{eq-suavVgamma-2}.
	\begin{proof}
		First assume that $V$ satisfies \eqref{eq-suavVgamma-2}. In particular, as we pointed out in the introduction $V(x)\leq C \rho^2(x)$.
		Now observe that our operator is the composition $T_{V^\gamma}\circ L^{-\gamma}$, where $T_{V^\gamma}$ means point-wise multiplication by $V^\gamma$. By Proposition~\ref{continuity}, $L^{-\gamma}$ is  bounded from $BMO_\rho^\beta$ into  $BMO_\rho^{\beta+ 2\gamma}$ when $\beta +2\gamma< 1$. 
		
		Next we claim that the operator of multiplication by $V^\gamma$ maps  $BMO_\rho^{\beta+ 2\gamma}$ into $BMO_\rho^\beta$. In fact, for any $f\in BMO_\rho^{\beta+2\gamma}$ and $|x-y|< \rho(x)$,
		\begin{equation}
		\begin{split}
		|V^\gamma(x)f(x)-V^\gamma(y)f(y)| & \leq |f(x)|\,|V^\gamma(x)-V^\gamma(y)+|V^\gamma(y)|\,|f(x)- f(y))|
		\\& \leq  C\rho^{\beta + 2\gamma}(x)\frac{|x-y|^{\beta }}{\rho^{2\gamma+\beta}(x)}+ C \frac{|x-y|^{\beta + 2\gamma }}{\rho^{2\gamma}(x)}
		\\& \leq C |x-y|^{\beta}.
		\end{split}  
		\end{equation}
		On the other side, for $x\in \mathbb{R}^d$ we have
		\begin{equation}
		|V^\gamma(x) f(x)| \leq \frac{C}{\rho^{2\gamma}(x)}\,\,\rho^{\beta + 2\gamma}(x)=C\, \rho^\beta(x),
		\end{equation}
		and the continuity $T_{V^\gamma}$ is proved. 
		
		Combining both results the $BMO_\rho ^\beta$-boundedness of $V^\gamma L^{-\gamma}$ follows as long as $\beta +2\gamma <1$.
		
		Now let us suppose that the operator is bounded on  $BMO_\rho ^\beta$. Lemma~\ref{ejemplos} allow us to evaluate the operator at the functions $f_x$. Hence, by means of \eqref{eq-averages}, we may write 
		\begin{equation}\label{eqaux}
		|V^\gamma(x)\left( L^{-\gamma}f_x\right) (x)| \leq C \|f_x\|_{BMO_\rho^\beta} \,\, \rho^\beta(x).
		\end{equation}

		Now we claim that there is a constant $C_1$, independent of $x$ such that
		\[
		|\left(  L^{-\gamma}f_x\right) (x)| \geq C_1 \rho (x)^{\beta+2\gamma}.
		\]
		If the claim were true, plugging this estimate into~\eqref{eqaux} we easily arrive to 
		\[
		V(x) \leq \frac{C}{\rho^2(x)}.
		\]

		 To get the lower bound for $\left( L^{-\gamma}f_x\right) (x)$, notice that by the definition of $f_x$ and for any value of $\lambda \leq 1$ to be chosen, we have
		\begin{equation}\label{eqaux3}
		|\left(  L^{-\gamma}f_x\right) (x)|\geq c_\beta \rho^{\beta+2\gamma}(x) \int_{B(x,\lambda \rho(x))} \JS_\gamma (x,y)\,dy.
		\end{equation}
		To estimate the last integral observe that, using Lemma \ref{compare} part~\eqref{item-disg-b}, we obtain for $|x-y|< \lambda \rho(x)$  
		\begin{equation}
		|\JS_\gamma(x,y)-\JC_\gamma(x,y)| \leq C_0 \lambda^{\delta_0} \frac{1}{|x-y|^{d-2\gamma}}.
		\end{equation}
		Therefore, since $\JC_\gamma$ is, upon a constant, $|x-y|^{2\gamma-d}$, we obtain
		\[
		\JS_\gamma(x,y)\geq                                                                                                          \JC_\gamma(x,y)-|\JS_\gamma(x,y)-\JC_\gamma(x,y)|\geq \frac{1}{2}\JC_\gamma(x,y),
		\]
		choosing $\lambda$ sufficiently small. Thus, for that value of $\lambda$,
		\[
		\int_{B(x,\lambda \rho(x))} \JS_\gamma (x,y)\,dy \geq C_2 \int _{B(x,\lambda \rho(x))} =\,\, C \rho^{2\gamma}(x),
		\]
		where the constant depends on $\beta$, $\gamma$ and the dimension, but not from $x$. Hence the claim is proved.
		
		Now, to get the condition  \eqref{eq-suavVgamma-2}, we use again that the operator is bounded on $BMO_\rho^\beta$  when applied to the functions $f_x$, this time through the smoothness condition \eqref{eq-smooth}, namely,
		\begin{equation}\label{eqaux2}
		|V^\gamma(x)\left( L^{-\gamma}f_x\right) (x)-V^\gamma(y)\left( L^{-\gamma}f_x\right) (y)| \leq C \|f_x\|_{BMO_\rho^\beta}\,\, |x-y|^\beta.
		\end{equation}
		Adding and subtracting the product $V^\gamma(y) \left( L^{-\gamma}f_x \right)(x)$, the left hand side is bounded from below by
		\[
		|V^\gamma(x)-V^\gamma(y)|\left|\left( L^{-\gamma}f_x\right)(x)\right|\, - V^\gamma(y)\,\,|\left( L^{-\gamma}f_x\right)(x)-\left( L^{-\gamma}f_x\right)(y)|= I-II.
		\]
		For the first term we use what we just proved for $L^{-\gamma} f_x (x)$. The second term can be estimated by above using the bound obtained for $V$ and the continuity $BMO_\rho^\beta-BMO^{\beta+2\gamma}_\rho$ of $L^{-\gamma}$ proved in Proposition \ref{continuity}.  Altogether, assuming $ |x-y|\leq \rho (x)$, we get
		\[
		I-II \geq C_0 \rho^{\beta+2\gamma}(x)|V^\gamma(x)-V^\gamma(y)|- \frac{C}{\rho^{2\gamma}(x)}\|f_x\|_{BMO_\rho^\beta}\,\, |x-y|^{\beta+2\gamma},
		\] 
		since $\rho(x)\simeq \rho(y)$.
		Therefore, going back to~\eqref{eqaux2},
		\[
		\rho^{\beta+2\gamma}(x)|V^\gamma(x)-V^\gamma(y)|\leq C \left( \frac{|x-y|^{\beta+2\gamma}}{\rho^{2\gamma}(x)} + |x-y|^\beta\right).
		\]
		Finally, using again $ |x-y|\leq \rho (x)$, we get 
		\[
		|V^\gamma(x)-V^\gamma(y)|\leq C \frac{|x-y|^\beta}{\rho^{\beta+2\gamma}(x)},
		\]
		as desired.
	\end{proof}
\end{prop}

\def\ocirc#1{\ifmmode\setbox0=\hbox{$#1$}\dimen0=\ht0 \advance\dimen0
	by1pt\rlap{\hbox to\wd0{\hss\raise\dimen0
			\hbox{\hskip.2em$\scriptscriptstyle\circ$}\hss}}#1\else {\accent"17 #1}\fi}
\def\cprime{$'$} \def\cprime{$'$}

\end{document}